\documentclass[final]{siamltex}

\usepackage{amsmath}
\usepackage{graphicx,graphics,txfonts}
\usepackage[ruled]{algorithm2e}
\usepackage{mathrsfs}
\usepackage[all]{xy}

\title{The conjugate gradient method with various viewpoints\thanks{The research was supported by
the National Natural Science Foundation of China (11171051) and the Fundamental Research Funds for the Central Universities (DUT18RC(4)067).}}

\author{Xuping Zhang\thanks{School of Mathematical Sciences,
Dalian University of Technology, Dalian, Liaoning 116025, P. R.
China ({\tt zhangxp@dlut.edu.cn, jiefei.yang01@gmail.com, zinglzy@163.com}).}
        \and Jiefei Yang$^\dag$
        \and Ziying Liu$^\dag$}

\begin{document}

\maketitle

\begin{abstract}
Connections of the conjugate gradient (CG) method with other methods in computational mathematics are surveyed, including the connections with the conjugate direction method, the subspace optimization method and the quasi-Newton method BFGS in numrical optimization, and the Lanczos method in numerical linear algebra. Two sequences of polynomials related to residual vectors and conjugate vectors are reviewed, where the residual polynomials are similar to orthogonal polynomials in the approximation theory and the roots of the polynomials reveal certain information of the coefficient matrix. The convergence rates of the steepest descent and CG are reconsidered in a viewpoint different from textbooks. The connection of infinite dimensional CG with finite dimensional preconditioned CG is also reviewed via numerical solution of an elliptic equation.
\end{abstract}



\pagestyle{myheadings} \thispagestyle{plain} \markboth{XUPING ZHANG,
JIEFEI YANG AND ZIYING LIU}{A REVIEW ON CONJUGATE GRADIENT METHOD}

\section{Introduction}
\label{sec:introduction}


The conjugate gradient method, proposed by Hestenes and Stiefel \cite{Hestenes-Stiefel}, is an effective method for solving linear system
\begin{align}
\label{eqn:Axb}
Ax=b,
\end{align}
where $A$ is symmetric positive definite. In the original paper \cite{Hestenes-Stiefel}, it is shown that the iterates in CG possess neat properties and that CG has connections with certain mathematical objects, such as orthogonal polynomials and continued fractions. From then on, many aspects of CG were explored with a large number of literatures, including the convergence rate \cite{convergencerate-VanderSluis,convergencerate-Sleijpen,convergencerate-Greenbaum}, preconditioners \cite{preconditon-Johnson,preconditon-Chan}, the related polynoials \cite{rootspolynomials-Manteuffel} and CG in Hilbert space \cite{CG-hilbert-Herzog-Sachs} and so on. Also, there are many works on exploring connections between CG and other algorithms, such as the quasi-Newton method BFGS \cite{BFGS-Nazareth,BFGS-Forsgren-Odland} and Lanczos method \cite{Lanczos-Meurant}. Such connections are interesting to the authors and become the main motivation of this paper. The purpose of this paper is to provide a further reading material for CG in textbooks.

%


In order to be a friendly further reading material for textbooks, the exposition
of the paper is made as detailed and easy to understand as possible. Usually claims are derived starting from elementary calculations. Only basic calculas, basic numerical linear algebra and basic Hilbert space theory are required as preliminary.

\section{Connections with other methods}
\label{set:connections}
There are several ways to derive CG, offering different viewpoints. In this paper, CG will be derived by the conjugate direction method, which is often done in textbooks. Connections of CG with other iterative methods are demonstrated by the fact that with the same initial guess, the iterates of CG are identical to the iterates generated by these methods.

\subsection{From the viewpoint of optimization}
Given a symmetric positive definite matrix $A$, solving $Ax=b$ is equivalent to minimizing $J(x)=\frac{1}{2}x^TAx-b^Tx$. Suppose there is an initial guess $x_0$. The corresponding residual $r_0=-\nabla J(x)=b-Ax_0$.
\begin{proposition}
At $x_k$, if $ J(x)$ is marching along the direction $d_k$, then optimal step size $\alpha_k$ and the quantity of descent is
\begin{align}
\label{eqn:descent-quantity}
\alpha_k = \frac{r_k^Td_k}{d_k^TAd_k},\quad
J(x_{k+1}) = J(x_k) - \frac{\left(r_k^Tp_k\right)^2}{2p_k^TAp_k},
\end{align}
where $r_k=b-Ax_k$ is the residual at $x_k$.
\end{proposition}

\subsubsection{Connection with conjugate direction method}
The CG method can be considered as a special case of conjugate direction method, just as mentioned in the original paper
\cite{Hestenes-Stiefel}. The first conjugate direction is taken as the steepest descent direction $r_0$, i.e., $p_0=r_0$. Then
\begin{align}
\label{eqn:x1}
x_1=x_0+\alpha_0 p_0,\quad \alpha_0=\frac{r_0^Tp_0}{p_0^TAp_0},\\
\label{eqn:r1}
r_1=b-Ax_1=r_0-\alpha_0Ap_0.
\end{align}
At $x_1$, instead of marching along the steepest descent direction $r_1$, a direction $p_1$ conjugate to $p_0$ is constructed, based on the information at hand. That is
\begin{align}
p_1=r_1+\beta_0 p_0,\quad \beta_0=-\frac{p_0^TAr_1}{p_0^TAp_0}.
\end{align}
Then
\begin{align}
x_2=x_1+\alpha_1 p_1,\quad \alpha_1=\frac{r_1^Tp_1}{p_1^TAp_1},\\
r_2=b-Ax_2=r_1-\alpha_1Ap_1.
\end{align}
Repeating this process, at $x_k$, gives the CG algorithm,
\begin{align}
\label{eqn:cd-iter-1}
& p_k=r_k+\beta_{k-1} p_{k-1}, \quad \beta_{k-1}=-\frac{p_{k-1}^TAr_k}{p_{k-1}^TAp_{k-1}},\\
\label{eqn:cd-iter-2}
& x_{k+1}=x_k+\alpha_k p_k, \quad \alpha_k=\frac{r_k^Tp_k}{p_k^TAp_k},\\
\label{eqn:cd-iter-3}
& r_{k+1}=b-Ax_k=r_k-\alpha_kAp_k.
\end{align}
The resulting directions $r_0,r_1,\dots,r_k$ are orthoganal to each other, and $p_0,p_1,\dots,p_k$ are conjugate to each other. By the orthogonal conditions and the conjugate conditions, there are alternative formulas for $\alpha_k$ and $\beta_k$,
\begin{align}
\label{eqn:cd-beta-alternative}
\beta_{k-1} & =-\frac{p_{k-1}^TAr_k}{p_{k-1}^TAp_{k-1}}=-\frac{p_{k-1}^TAr_k}{p_{k-1}^TAr_{k-1}}=-\frac{(r_{k-1}-r_k)^Tr_k}{(r_{k-1}-r_k)^Tr_{k-1}}=\frac{r_k^Tr_k}{r_{k-1}^Tr_{k-1}},\\
\label{eqn:cd-alpha-alternative}
\alpha_k & =\frac{r_k^Tr_k}{p_k^TAp_k}=\frac{p_k^Tr_k}{p_k^TAp_k}=\frac{p_k^Tr_0}{p_k^TAp_k}.
\end{align}
If $x_0=0$ is chosen, then $r_0=b$ and $\alpha_k=\frac{p_k^Tb}{p_k^TAp_k}$. And if the process terminates at $k=n$, then
\begin{align}
\label{eqn:cd-exact-solution}
x = \alpha_0 p_0 + \alpha_1 p_1 + \ldots + \alpha_{n-1} p_{n-1} = \sum_{i=0}^{n-1}\frac{p_i^Tb}{p_i^TAp_i}p_i = \left(\sum_{i=0}^{n-1}\frac{p_i p_i^T}{p_i^TAp_i}\right) b.
\end{align}
Therefore an explicit expression of $A^{-1}$ is resulted
\begin{align}
\label{eqn:cd-A-inverse}
A^{-1} = \sum_{i=0}^{n-1}\frac{p_i p_i^T}{p_i^TAp_i}.
\end{align}

\subsubsection{Connection with subspace optimization}
The CG method can be considered as a two-dimensional subspace minimization method for the objective function $ J(x)$, see for example \cite{subspaceoptimization-Xu}. At the beginning of subspace minimization, there is only one direction $r_0$ availabe. Therefore we proceed as (\ref{eqn:x1})-(\ref{eqn:r1}) to obtain $r_1$. Now consider minimizing $ J(x)$ in the two-dimensional affine subspace $\pi_2=x_1+\mbox{span}\{r_1,r_0\}$. In order to solve the subproblem, let $x=x_1+\xi r_1+\eta r_0$ and consider the function $h(\xi,\eta)= J(x_1+\xi r_1+\eta r_0)$, where $(\xi,\eta)\in \mathbb{R}^2$. Forcing $\frac{\partial h}{\partial \xi}=0$ and $\frac{\partial h}{\partial \eta}=0$, gives
\begin{align}
\label{eqn:subspace-eqn1}
r_1^TAr_1\cdot\xi + r_1^TAr_0\cdot\eta & = r_1^Tr_1,\\
\label{eqn:subspace-eqn2}
r_1^TAr_0\cdot\xi + r_0^TAr_0\cdot\eta & = 0,
\end{align}
where the orthoganality $r_1^Tr_0=0$ is used. From (\ref{eqn:subspace-eqn1})-(\ref{eqn:subspace-eqn2}), there holds
\begin{align}
\label{eqn:subspace-soln1}
\eta & = -\frac{r_1^TAr_0}{r_0^TAr_0}\xi,\\
\label{eqn:subspace-soln2}
\xi & = \frac{r_1^Tr_1}{r_1^TAr_1-(r_1^TAr_0)^2/r_0^TAr_0}.
\end{align}
Denote $\tilde{p}_1=r_1+\frac{\eta}{\xi}r_0=r_1-\frac{r_1^TAr_0}{r_0^TAr_0}r_0$. It is easy to see that $\tilde{p}_1=p_1$, the $p_1$ in the CG method. Thus $\tilde{p}_1^TAr_0=0$ as in the CG method. A little calculation shows that
\begin{align}
\xi & = \frac{r_1^Tr_1}{r_1^TAr_1-(r_1^TAr_0)^2/r_0^TAr_0}=\frac{r_1^Tp_1}{p_1^TAp_1}=\alpha_1.
\end{align}
where $\alpha_1$ is the step size in the CG method.
The minimizer of the subproblem is
\begin{align}
\nonumber
\tilde{x}_2 & = x_1 + \xi\left(r_1+\frac{\eta}{\xi}r_0\right) = x_1 + \alpha_1 p_1.
\end{align}
From the above formula, it can be seen that $\tilde{x}_2=x_2$ and that $x_2$ is the minimizer of $J(x)$ both in the two-dimensional affine subspace $\pi_2$ and in the direction $p_1$. The next residual vector arises
$$
r_2=b-Ax_2=r_1-\frac{r_1^Tp_1}{p_1^TAp_1}Ap_1,
$$
which satisfies orthogonal conditions $r_2^Tr_1=0$ and $r_2^Tp_1=0$. Now there are two options for the new two-dimensional affine subspace for $J(x)$ to be minimized, i.e., $\pi_2=x_1+\mbox{span}\{r_2,r_1\}$ and $\pi_2=x_1+\mbox{span}\{r_2,p_1\}$. Which subspace should be chosen? In order to answer this question, consider a more general subproblem, i.e., at $x_k$, to minimize $ J(x)$ in a two dimensional subspace $\mbox{span}\{u,v\}$, with orthogonal condition $u^Tv=0$,
\begin{align}
\label{eqn:general-subproblem}
\min_{\xi,\eta} J(x_k+\xi u+\eta v).
\end{align}
By similar arguments as the case $\pi_2=x_1+\mbox{span}\{r_1,r_0\}$, the following proposition can be derived.
\begin{proposition}
\label{proposi:general-subproblem-orth}
The minimizer $\tilde{x}_{k+1}=x_k+\xi u+\eta v$ of $J(x_k+\xi u+\eta v)$ is also the minimizer of $ J(x)$ at $x_k$ along the direction $\tilde{p} = u+\frac{\eta}{\xi}v$, with step size $\tilde{\alpha} = \xi$,
where
\begin{align}
\frac{\eta}{\xi} = -\frac{u^TAv}{v^TAv},\quad \xi =\frac{u^Tu}{p^TAp}.
\end{align}
Furthermore, the new residual $\tilde{r}_{k+1}=b-A\tilde{x}_{k+1}$ is orthogonal to $u$, $v$ and $\tilde{p}$ is conjugate to $v$, i.e.,
\begin{align}
\label{eqn:general-subproblem-orth}
\tilde{r}_{k+1}^Tu=0,\quad \tilde{r}_{k+1}^Tv=0,\quad \tilde{p}^TAv=0.
\end{align}
\end{proposition}

Applying Propositon \ref{proposi:general-subproblem-orth} at $x_2$, if $\pi_2=x_2+\mbox{span}\{r_2,r_1\}$ is chosen, i.e., $u=r_2,v=r_1$, then $\tilde{r}_3$ is orthogonal to $r_2$ and $r_1$ but not necessarily orthogonal to $r_0$, and $\tilde{p}_2$ is conjugate to $r_1$ rather to $p_1$. Hence, the works that we have done seems to be messed up.

If $\pi_2=x_2+\mbox{span}\{r_2,p_1\}$ is chosen, i.e., $u=r_2,v=p_1$, then $\tilde{r}_3$ is orthogonal to $r_2$ and $p_1$, and $\tilde{p}_2$ is conjugate to $p_1$. It can be verified that such $\tilde{p}_2$ is also conjugate to $p_0$ and therefore $\tilde{r}_3$ is orthogonal to $r_0$. In fact, with $u=r_2,v=p_1$
\begin{align}
\tilde{p}_2=r_2-\frac{p_{1}^TAr_1}{p_{1}^TAp_{1}} p_{1}=p_2,\\
\tilde{x}_{3}=x_2+\frac{r_2^Tp_2}{p_2^TAp_2} p_2=x_3,\\
\tilde{r}_{3}=b-Ax_3=r_2-\frac{r_2^Tp_2}{p_2^TAp_2}Ap_2=r_3,
\end{align}
where $p_2$, $x_3$ and $r_3$ are iterates of the CG method. Therefore $\tilde{p}_2$ is conjugate to $p_0$ and $\tilde{r}_3$ is orthogonal to $r_0$. These resulting properties seem satifying, because they gurantee that we are making progress.

Repeating this process, at $x_k$, the subspace $\pi_2=x_k+\mbox{span}\{r_k,p_{k-1}\}$ is taken, which results in
\begin{align}
\label{eqn:subspace-iter-1}
\tilde{p}_k=r_k-\frac{p_{k-1}^TAr_k}{p_{k-1}^TAp_{k-1}} p_{k-1},\\
\label{eqn:subspace-iter-2}
\tilde{x}_{k+1}=x_k+\frac{r_k^Tp_k}{p_k^TAp_k} p_k,\\
\label{eqn:subspace-iter-3}
\tilde{r}_{k+1}=b-Ax_k=r_k-\frac{r_k^Tp_k}{p_k^TAp_k}Ap_k,
\end{align}
which are the same iterates as $p_k$, $x_{k+1}$, $r_{k+1}$ in the CG method (\ref{eqn:cd-iter-1})-(\ref{eqn:cd-iter-3}).

\subsubsection{Connection with BFGS}
BFGS is a special quasi-Newton method for minimizing a function $h(x)$. Denote $g_k=\nabla h(x_k)$, $s_k=x_{k+1}-x_k$, and $y_k=g_{k}-g_{k+1}$. Given the initial data $x_0$, $H_0$, the BFGS method reads
\begin{align}
\label{eqn:bfgs-1}
d_k & = -H_k g_k,\\
\label{eqn:bfgs-2}
x_{k+1} & = x_k + \alpha_k d_k,\quad \alpha_k=\mbox{arg}\min_{\alpha}h(x_k+\alpha d_k),\\
\label{eqn:bfgs-3}
H_{k+1} & = H_{k}+\frac{1}{s_{k}^Ty_{k}}\left[1+\frac{y_{k}^TH_{k}y_{k}}{s_{k}^Ty_{k}}\right]s_{k}s_{k}^T - \frac{1}{s_{k}^Ty_{k}}(s_{k}y_{k}^TH_{k}+H_{k}y_{k}s_{k}^T).
\end{align}
When applied to the quadratic function $ J(x)$, with the same initial guess $x_0$ as the CG method, and with $H_0 = I$, BFGS will produce the same iterates as CG, maybe firstly discovered by Nazareth \cite{BFGS-Nazareth}. Noting that the residual $r=b-Ax=-\nabla  J(x)$, BFGS holds
\begin{align}
\hat{p}_0 & = -H_0 g_0 = r_0,\\
\hat{x}_{1} & = x_0 + \alpha_0 \hat{p}_0,\quad \alpha_0=\frac{r_0^T\hat{p}_0}{\hat{p}_0^TA\hat{p}_0},\\
H_{1} & = H_{0}+\frac{1}{s_{0}^Ty_{0}}\left[1+\frac{y_{0}^TH_{0}y_{0}}{s_{0}^Ty_{0}}\right]s_{0}s_{0}^T - \frac{1}{s_{0}^Ty_{0}}(s_{0}y_{0}^TH_{0}+H_{0}y_{0}s_{0}^T).
\end{align}
Thus $\hat{x}_1$ is identical to $x_1$ generated by the CG method, and $\hat{p}_0$ identical to $p_0$. In passing, $r_1$ is the same, with orthogonal condition $r_1^Tr_0=0$. Note that $s_0=x_1-x_0=\alpha_0 \hat{p}_0$, $y_0=g_{0}-g_{1}=r_1-r_0=\alpha_0 A\hat{p}_0$ and $s_0^Tr_1=0$. The next marching direction $\hat{p}_1$ is computed as
\begin{align}
\label{eqn:bfgs-cd-1}
\hat{p}_1 & = -H_1 g_1 = H_1 r_1 = r_1 - \frac{y_0^TH_0r_1}{s_0^Ty_0}s_0 = r_1 - \frac{y_0^Tr_1}{\hat{p}_0^Ty_0}\hat{p}_0 = r_1 - \frac{p_0^TAr_1}{p_0^TAp_0}p_0 = p_1.
\end{align}
The above formula shows that $\hat{p}_1$ is identical to $p_1$ of the CG method. Therefore all the orthogonal conditions and conjugate conditions pass up to $r_2$ and $\hat{p}_1$. Furthermore,
\begin{align}
\label{eqn:bfgs-cd-2}
H_1r_2 = H_0r_2 - \frac{1}{s_{0}^Ty_{0}}s_{0}y_{0}^TH_{0}r_2 = H_0r_2 - \frac{1}{s_{0}^Ty_{0}}s_{0}(r_1-r_0)^Tr_2 = H_0r_2 = r_2.
\end{align}

Suppose that up to $x_{k}$, the iterates and directions of BFGS and the CG method are the same, with $s_{k-1}^Tr_k=0$ and $H_{k-1}r_{k}=H_{k-2}r_{k}=\cdots=H_{0}r_{k}=r_{k}$. The next BFGS direction $\hat{p}_k$ is computed as
\begin{align}
\label{eqn:bfgs-cd-3}
\nonumber
\hat{p}_k & = -H_k g_k = H_k r_k = H_{k-1}r_k - \frac{1}{s_{k-1}^Ty_{k-1}}s_{k-1}y_{k-1}^TH_{k-1}r_{k} = r_k - \frac{1}{s_{k-1}^Ty_{k-1}}s_{k-1}y_{k-1}^Tr_{k}\\
& = r_k - \frac{y_{k-1}^Tr_{k}}{s_{k-1}^Ty_{k-1}}s_{k-1} = r_k - \frac{\hat{p}_{k-1}^TAr_{k}}{\hat{p}_{k-1}^TA\hat{p}_{k-1}}\hat{p}_{k-1} = r_k - \frac{p_{k-1}^TAr_{k}}{p_{k-1}^TAp_{k-1}}p_{k-1} = p_k,
\end{align}
where $p_k$ is the conjugate vector in the CG method. Therefore $\hat{x}_{k+1}=x_{k+1}$ will hold. Furthermore, similar to (\ref{eqn:bfgs-cd-2}) it can be derived that $H_{k}r_{k+1}=H_{k-1}r_{k+1}=\cdots=H_{0}r_{k+1}=r_{k+1}$. Hence the induction method applies. For more information on the connection of CG and quasi-Newton methods, see \cite{BFGS-Forsgren-Odland}.

\subsection{From the viewpoint of linear algebra}
Recall the Lanczos process:
given $x_0$, compute $r_0=b-Ax_0$, $v_0=\frac{r_0}{\|r_0\|}$, $\tau_0=\|r_0\|$.
\begin{align}
\label{eqn:lanczos-1}
\tau_1 v_1 & = Av_0 - \sigma_0 v_0,\quad \sigma_0=v_0^TAv_0,\quad \tau_1=\|Av_0 - \sigma_0 v_0\|,\\
\label{eqn:lanczos-2}
\tau_2 v_2 & = Av_1 - \sigma_1 v_1 - \tau_1 v_0,\quad \sigma_1=v_1^TAv_1,\quad \tau_2=\|Av_1 - \sigma_1 v_1\|,\\
\nonumber
\vdots\\
\label{eqn:lanczos-vk}
\tau_k v_k & = Av_{k-1} - \sigma_{k-1} v_{k-1} - \tau_{k-1} v_{k-2},\quad \sigma_{k-1}=v_{k-1}^TAv_{k-1},\\
\nonumber
& \quad \tau_k=\|Av_{k-1} - \sigma_{k-1} v_{k-1} - \tau_{k-1} v_{k-2}\|.
\end{align}
It is easy to see that $v_i$ and $v_j$ are orthogonal, if $i\neq j$.
Denote $V_k=[v_0,v_1,\cdots,v_{k-1}]$, $\mathcal{V}_k=\mbox{span}\{v_0,v_1,\cdots,v_{k-1}\}$, $e_k\in \mathbb{R}^k$ with the $k$-th component being 1, and
\begin{equation}
T_k\triangleq \left(
  \begin{array}{ccccc}
    \sigma_0 & \tau_1 &  & & \\
    \tau_1 & \sigma_1 & \tau_2 & & \\
     & \ddots & \ddots & \ddots &\\
     & & \ddots & \ddots & \tau_{k-1}\\
     & & & \tau_{k-1} & \sigma_{k-1}\\
  \end{array}
\right). \nonumber
\end{equation}
Then
\begin{align}
\label{eqn:lanczos-3}
AV_k & = V_k T_k + \tau_k v_k e_k^T,\\
V_k^TAV_k & = T_k.
\end{align}
$T_k$ is the projection to $\mathcal{V}_k$ of the restriction $A|_{\mathcal{V}_k}$.

Consider the projected equations
\begin{align}
\label{eqn:lanczos-4}
V_k^TA(x_0+V_k z_k) = V_k^T b,
\end{align}
i.e.,
\begin{align}
\label{eqn:lanczos-5}
V_k^TAV_k z_k = V_k^T r_0 = \|r_0\|e_1,\mbox{ or }T_k z_k = \tau_0 e_1.
\end{align}

Solve $T_k z_k = \|r_0\|e_1$ by the Cholesky method. Suppose the Cholesky decomposition of $T_k$ is
\begin{align}
\label{eqn:lanczos-6}
T_k = L_kD_kL_k^T.
\end{align}
Then
\begin{align}
\nonumber
z_k = L_k^{-T}D_k^{-1}L_k^{-1}\tau_0 e_1,
\end{align}
and
\begin{align}
\nonumber
\bar{x}_k = x_0+V_k z_k.
\end{align}
Note that although the columns of $V_k$ is accumulated with each increase in $k$, i.e., $V_{k}=[V_{k-1}, v_k]$, the components of such $z_k$ changes fully with each increase in $k$. The reason is that $T_k$ is not a lower triangular matrix.
To overcome this difficulty, rearrange the factors in the expression of $\bar{x}_k$,
\begin{align}
\nonumber
\bar{x}_k = x_0+V_k z_k = x_0+V_k L_k^{-T}D_k^{-1}L_k^{-1}\tau_0 e_1 = x_0+\bar{W}_k w_k,
\end{align}
where $\bar{W}_k=V_k L_k^{-T}$ and $w_k=D_k^{-1}L_k^{-1}\tau_0 e_1$. Let $\bar{W}_k=[\bar{p_0},\ldots,\bar{p}_{k-1}]$.

\begin{proposition}
\label{proposi:accumulativeness}
$L_k, D_k, \bar{W}_k, w_k$ are all accumulated with each increase in $k$, i.e.,
\begin{align}
\nonumber
& L_{k}(1:k-1,1:k-1) =L_{k-1},& \quad D_{k}(1:k-1,1:k-1) = D_{k-1},\\
\nonumber
& \bar{W}_{k} =[\bar{W}_{k-1}, \bar{p}_{k}],& \quad w_k = \left(w_{k-1},\bar{\alpha}_k\right)^T,
\end{align}
where $M(1:k-1,1:k-1)$ means the first $k-1$ rows and first $k-1$ columns of a matrix $M$. The columns of $\bar{W}_k$ are conjugate, i.e.,
\begin{align}
\label{eqn:lanczos-7}
\bar{p}_i^TA\bar{p}_j=0,\quad i\neq j
\end{align}
and
\begin{align}
\label{eqn:lanczos-8}
\bar{p}_i^TA\bar{p}_i=\delta_i,
\end{align}
where $\delta_i$ is the $i$-th element of the diagonal of $D_k$.
\end{proposition}

\begin{proof}
The claim for $L_k$ and $D_k$ is obvious since $T_k$ is accumulated with each increase in $k$. Now consider the claim for $\bar{W}_{k}$. Rewrite $\bar{W}_k=V_k L_k^{-T}$ as $\bar{W}_k L_k^{T}= V_k$, or
\begin{equation}
\left[ \bar{p}_0,\ldots,\bar{p}_{k-1} \right]
\left(
  \begin{array}{ccccc}
    1 & l_{21} &  & & \\
     & 1 & l_{32} & & \\
     &  & \ddots & \ddots &\\
     & &  & \ddots & l_{k-1,k}\\
     & & &  & 1\\
  \end{array}
\right)
=
\left[ v_0,\ldots,v_{k-1} \right] . \nonumber
\end{equation}
By the above relation, $\bar{W}_{k}$ is accumulated with each increase in $k$. The argument for $w_k$ is similar.
\end{proof}

By Proposition \ref{proposi:accumulativeness},
\begin{align}
\label{eqn:lanczos-iterate}
\bar{x}_k = x_0+\bar{W}_k w_k = x_0+\bar{W}_{k-1} w_{k-1}+\bar{\alpha}_k \bar{p}_k=\bar{x}_{k-1}+\bar{\alpha}_k \bar{p}_k.
\end{align}

Up to now, with the orthogonal conditions of $v_i$, the conjugate conditions of $p_i$ and the iterates (\ref{eqn:lanczos-iterate}), it can be seen that the Lanczos + Cholesky method is quite similar to the CG method. The connection can be found in \cite{Lanczos-Householder,Lanczos-Meurant}, and will be explored in detail in the following.

Suppose that the two methods start with the same initial guess $x_0$. At $k=0$, $T_0=\sigma_0$, therefore $L_0=1$, $D_0=\sigma_0$, and $\bar{p}_0=v_0$, $\bar{\alpha}_0=D_0^{-1}L_0^{-1}\tau_0=\frac{\tau_0}{\sigma_0}$. Note that $\sigma_0=v_0^TAv_0=\frac{r_0^TAr_0}{\|r_0\|^2}=\frac{1}{\alpha_0}$, where $\alpha_0$ is the first step size in the CG method. There holds
\begin{align}
\label{eqn:lanczos-connection-p0}
\bar{p}_0=v_0=\frac{r_0}{\|r_0\|}=\frac{p_0}{\|r_0\|},\\
\label{eqn:lanczos-connection-alpha0}
\bar{\alpha}_0=\frac{\tau_0}{\sigma_0}=\|r_0\|\alpha_0.
\end{align}
Thus
\begin{align}
\label{eqn:lanczos-connection-alpha0-p0}
\bar{\alpha}_0\bar{p}_0=\alpha_0 p_0,\\
\label{eqn:lanczos-connection-x1}
\bar{x}_1=x_0+\bar{\alpha}_0\bar{p}_0=x_1,
\end{align}
where $p_0$ and $x_1$ are the iterates in the CG method. In passing, $\bar{r}_1=r_1$. Rewrite $r_1$ in terms of $v_0$,
\begin{align}
\label{eqn:lanczos-connection-r1}
r_1 = r_0 - \bar{\alpha}_0A\bar{p}_0=\tau_0 v_0 - \bar{\alpha}_0Av_0=-\bar{\alpha}_0\left(Av_0-\frac{\tau_0}{\bar{\alpha}_0}v_0\right)=-\bar{\alpha}_0\left(Av_0-\sigma_0v_0\right).
\end{align}
It is clear now that
\begin{align}
\nonumber
-\frac{\|r_1\|}{\bar{\alpha}_0} \frac{r_1}{\|r_1\|} = Av_0-\sigma_0v_0=\tau_1 v_1.
\end{align}
Since $\tau_1>0$ and $v_1$ is a unit vector, it must be that
\begin{align}
\label{eqn:lanczos-connection-r1v1}
v_1=-\frac{r_1}{\|r_1\|}.
\end{align}
Therefore
\begin{align}
\nonumber
\sigma_1 & =v_1^TAv_1=\frac{r_1^TAr_1}{\|r_1\|^2}=\frac{p_1^TAp_1+\beta_0^2p_0^TAp_0}{\|r_1\|^2}\\
\label{eqn:lanczos-connection-sigma-1}
& = \frac{p_1^TAp_1}{\|r_1\|^2}+\beta_0\frac{\|r_1\|^2}{\|r_0\|^2}\frac{p_0^TAp_0}{\|r_1\|^2}=\frac{1}{\alpha_1}+\frac{\beta_0}{\alpha_0},\\
\label{eqn:lanczos-connection-tau-1}
\tau_1 & =\frac{\|r_1\|}{\bar{\alpha}_0}=\frac{\|r_1\|}{\|r_0\|\alpha_0}=\frac{\sqrt{\beta_0}}{\alpha_0},
\end{align}
where $\beta_0$ is the one in the CG method.

Let the indices of the matrices $L_k$ and $D_k$ be labeled starting from 0. By the relation $\bar{W}_2=V_2L_2^{-T}$, there holds
\begin{align}
\label{eqn:lanczos-connection-pv1}
[\bar{p}_0,\bar{p}_1]
\left[
  \begin{array}{cc}
    1 & l_{10}\\
     & 1
  \end{array}
\right]=[v_0,v_1]
\end{align}
and thus
\begin{align}
\nonumber
l_{10}\bar{p}_0 + \bar{p}_1 = v_1.
\end{align}
Since $\bar{p}_0$ and $\bar{p}_1$ are conjugate, by (\ref{eqn:lanczos-connection-p0}) and (\ref{eqn:lanczos-connection-r1v1})
\begin{align}
\nonumber
l_{10} & =\frac{v_1^TA\bar{p}_0}{\bar{p}_0^TA\bar{p}_0}=\frac{\|r_0\|^2}{p_0^TAp_0}\left(-\frac{r_1}{\|r_1\|}\right)^TA\left(\frac{p_0}{\|r_0\|}\right)\\
\label{eqn:lanczos-connection-l10}
& =-\frac{r_1^T(\alpha_0Ap_0)}{\|r_1\|\|r_0\|}
=\frac{r_1^T(r_1-r_0)}{\|r_1\|\|r_0\|}
=\frac{\|r_1\|}{\|r_0\|}.
\end{align}
Therefore
\begin{align}
\nonumber
\bar{p}_1 & = v_1 - l_{10}\bar{p}_0=-\frac{r_1}{\|r_1\|}-\frac{\|r_1\|}{\|r_0\|}\frac{p_0}{\|r_0\|} = - \frac{1}{\|r_1\|}\left(r_1-\frac{\|r_1\|^2}{\|r_0\|^2}p_0\right)\\
\label{eqn:lanczos-connection-p1}
& = - \frac{1}{\|r_1\|}\left(r_1-\beta_0 p_0\right)=- \frac{1}{\|r_1\|}p_1,
\end{align}
where $p_1$ and $x_2$ are the iterates in the CG method. Hence by the above equality,
\begin{align}
\label{eqn:lanczos-connection-d1}
\delta_1 = \bar{p}_1^TA\bar{p}_1 = \frac{p_1^TAp_1^T}{r_1^Tr_1}=\frac{1}{\alpha_1},
\end{align}
where $\alpha_1$ is the second step size in the CG method.

By the relation $w_2=D_2^{-1}L_2^{-1}\tau_0 e_1$, there holds
\begin{align}
\label{eqn:lanczos-connection-alphas}
\left[
  \begin{array}{cc}
    1 & \\
    l_{10} & 1
  \end{array}
\right]
\left[
  \begin{array}{c}
    \delta_0 \bar{\alpha}_0\\
    \delta_1 \bar{\alpha}_1
  \end{array}
\right]=
\tau_0\left[
  \begin{array}{c}
    1\\
    0
  \end{array}
\right],
\end{align}
and thus by (\ref{eqn:lanczos-connection-d1})
\begin{align}
\label{eqn:lanczos-connection-alpha1}
\bar{\alpha}_1=-\frac{l_{10}\delta_0\bar{\alpha}_0}{\delta_1}=-\frac{\|r_1\|}{\|r_0\|}\frac{\|r_0\|}{\delta_1}=-\|r_1\|\alpha_1.
\end{align}
From (\ref{eqn:lanczos-connection-p1}) and (\ref{eqn:lanczos-connection-alpha1}), there holds
\begin{align}
\label{eqn:lanczos-connection-alpha1-p1}
\bar{\alpha}_1\bar{p}_1=\alpha_1 p_1,\\
\label{eqn:lanczos-connection-x2}
\bar{x}_2=x_1+\bar{\alpha}_1\bar{p}_1=x_2,
\end{align}
where $x_2$ are the iterates in the CG method. By induction, the connection can be established,
\begin{align}
\label{eqn:lanczos-connection-rvk}
v_k & = (-1)^k\frac{r_k}{\|r_k\|},\\
\label{eqn:lanczos-connection-pk}
\bar{p}_k & = (-1)^k\frac{1}{\|r_k\|}p_k,\\
\label{eqn:lanczos-connection-alphak}
\bar{\alpha}_k & =(-1)^k\|r_k\|\alpha_k,\\
\label{eqn:lanczos-connection-sigma-k}
\sigma_k & = \frac{1}{\alpha_k}+\frac{\beta_{k-1}}{\alpha_{k-1}},\\
\label{eqn:lanczos-connection-tau-k}
\tau_k & = \frac{\sqrt{\beta_{k-1}}}{\alpha_{k-1}},\\
\label{eqn:lanczos-connection-lkkm1}
l_{k,k-1} & = \frac{\|r_k\|}{\|r_{k-1}\|},\\
\label{eqn:lanczos-connection-dk}
\delta_k & = \frac{1}{\alpha_k}.
\end{align}

If the process terminates at $k=n$, then $V_{n}$ is an orthogonal matrix and therefore
\begin{align}
\label{eqn:lanczos-det}
\det(A) = \det(T_n) = \det(D_{n}) = \prod_{k=0}^{n-1}\frac{1}{\alpha_k}, \mbox{ or }, \det(A^{-1}) = \prod_{k=0}^{n-1}\alpha_k.
\end{align}
On the other hand, since $\beta_{i}=\frac{\|r_{i+1}\|^2}{\|r_{i}\|^2}$, the product of $\beta_{i}$,
\begin{align}
\label{eqn:lanczos-product-beta}
\prod_{i=0}^{k-1}\beta_i=\frac{\|r_k\|^2}{\|r_{0}\|^2}
\end{align}
is nothing but the square of the 2-norm of the relative residual.

\section{Residual polynomials and conjugate polynomials}
In the CG iteration process, there are two sequences of vectors $\{r_k\}$ and $\{p_k\}$, i.e.,
those of residual vectors and conjugate vectors, with both $r_k$ and $p_k$ lying in the Krylov subspace
$\mathcal{K}=\mbox{span}\{r_0,Ar_0,\ldots,A^{k-1}r_0\}$. As is noted in the original paper \cite{Hestenes-Stiefel}, checking the CG process, $\{r_k\}$ and $\{p_k\}$ are related to two sequences of polynomials
$\{R_k(\lambda)\}$ and $\{P_k(\lambda)\}$ such that
\begin{align}
r_k = R_k(A)r_0,\\
p_k = P_k(A)r_0.
\end{align}
In this sense, $\{R_k(\lambda)\}$ and $\{P_k(\lambda)\}$ will be called residual polynomials and conjugate polynomials respectively in this paper.
Note that the degrees of $R_k(\lambda)$ and $P_k(\lambda)$ are both $k$. Although $\{r_k\}$ and $\{p_k\}$ are intertwined as
\begin{align}
r_{k+1} = r_{k} - \alpha_k A p_{k},\\
p_{k+1} = r_{k+1} + \beta_k p_{k},
\end{align}
They can be decoupled to obtain their own three-term reccurences as,
\begin{align}
\nonumber
r_{k+1} & = r_{k} - \alpha_k A p_{k}\\
\nonumber
& = r_{k} - \alpha_k A \left( r_k+\beta_{k-1}p_{k-1} \right)\\
\nonumber
& = (1-\alpha_k A)r_{k} - \alpha_k \beta_{k-1} A p_{k-1}\\
\nonumber
& = (1-\alpha_k \lambda)r_{k} - \alpha_k \beta_{k-1} \frac{1}{\alpha_{k-1}}
\left( r_{k-1}-r_{k} \right)\\
\label{eqn:three-term-reccurence-rk}
& = \left( 1+\frac{\alpha_k}{\alpha_{k-1}}\beta_{k-1}-\alpha_k A \right) r_{k} - \frac{\alpha_k}{\alpha_{k-1}}\beta_{k-1} r_{k-1},\\
\nonumber
p_{k+1} & = r_{k+1} + \beta_k p_{k}\\
\nonumber
& = r_{k} - \alpha_k A p_k + \beta_k p_{k}\\
\nonumber
& = p_{k} -\beta_{k-1}p_{k-1} - \alpha_k A p_k + \beta_k p_{k}\\
\label{eqn:three-term-reccurence-pk}
& = \left( 1+\beta_k- \alpha_k A \right) p_{k} -\beta_{k-1}p_{k-1}.
\end{align}
Correspondingly, $\{R_k(\lambda)\}$ and $\{P_k(\lambda)\}$ are intertwined as
\begin{align}
R_{k+1}(\lambda) = R_{k}(\lambda) - \alpha_k \lambda P_{k}(\lambda),\\
P_{k+1}(\lambda) = R_{k+1}(\lambda) + \beta_k P_{k}(\lambda).
\end{align}
and have their own three-term reccurences as,
\begin{align}
\nonumber
R_{k+1}(\lambda) & = \left( 1+\frac{\alpha_k}{\alpha_{k-1}}\beta_{k-1}-\alpha_k \lambda \right) R_{k}(\lambda) - \frac{\alpha_k}{\alpha_{k-1}}\beta_{k-1} R_{k-1}(\lambda)\\
P_{k+1}(\lambda) & = \left( 1+\beta_k- \alpha_k \lambda \right) P_{k}(\lambda) -\beta_{k-1}P_{k-1}(\lambda).
\end{align}
The properties of the polynomials $\{R_k(\lambda)\}$ and $\{P_k(\lambda)\}$ reveals certain information of the matrix $A$.

\subsection{The roots of residual polynomials and conjugate polynomials}
Recall the Lanczos process. Two sequences of polynomials $\{\bar{R}_k(\lambda)\}$ and $\{\bar{P}_k(\lambda)\}$ can also be defined such that
\begin{align}
v_k = \bar{R}_k(A)v_0,\\
\bar{p}_k = \bar{P}_k(A)v_0.
\end{align}
Due to the correspondences of $v_k,r_k$ and $\bar{p}_k,p_k$ in (\ref{eqn:lanczos-connection-rvk}) and (\ref{eqn:lanczos-connection-pk}), $\bar{R}_k(\lambda)$ is
a multiple of $R_k(\lambda)$, and $\bar{P}_k(\lambda)$ a multiple of $P_k(\lambda)$. Therefore the roots of $R_k(\lambda)$ are the same as those of $\bar{R}_k(\lambda)$, and the roots of $P_k(\lambda)$ the same as those of $\bar{P}_k(\lambda)$. The roots of $\bar{R}_k(\lambda)$ and those of $\bar{P}_k(\lambda)$ are closely related to the the relation $AV_k = V_kT_k + \tau_k v_k e_k^T$ in Lanczos process. Rewrite this relation as $V_k^TA = T_kV_k^T + \tau_k e_k v_k^T$, that is
\begin{align}
\nonumber
\left(
  \begin{array}{c}
    v_0^T \\
    v_1^T \\
    \vdots \\
    v_{k-1}^T
  \end{array}
\right)A
= T_k
\left(
  \begin{array}{c}
    v_0^T \\
    v_1^T \\
    \vdots \\
    v_{k-1}^T
  \end{array}
\right)
+ \tau_k
\left(
  \begin{array}{c}
    0 \\
    0 \\
    \vdots \\
    v_{k}^T
  \end{array}
\right).
\end{align}
Note that $v_k=\bar{R}_k(A)v_0$, the above formula is rewritten further as
\begin{align}
\label{eqn:reccurence-Tk-Rk}
\left(
  \begin{array}{c}
    \left( \bar{R}_0(A)v_0 \right)^T \\
    \left( \bar{R}_1(A)v_0 \right)^T \\
    \vdots \\
    \left( \bar{R}_{k-1}(A)v_0 \right)^T \\
  \end{array}
\right)A
= T_k
\left(
  \begin{array}{c}
    \left( \bar{R}_0(A)v_0 \right)^T \\
    \left( \bar{R}_1(A)v_0 \right)^T \\
    \vdots \\
    \left( \bar{R}_{k-1}(A)v_0 \right)^T \\
  \end{array}
\right)
+ \tau_k
\left(
  \begin{array}{c}
    0 \\
    0 \\
    \vdots \\
    \left( \bar{R}_{k}(A)v_0 \right)^T
  \end{array}
\right).
\end{align}
It turns out that the sequence of polynomials $\bar{R}_k(\lambda)$ satisfies the same reccurence relation,
\begin{align}
\label{eqn:eigenvalue-Tk}
\left(
  \begin{array}{c}
    \bar{R}_0(\lambda) \\
    \bar{R}_1(\lambda) \\
    \vdots \\
    \bar{R}_{k-1}(\lambda) \\
  \end{array}
\right)\lambda
= T_k
\left(
  \begin{array}{c}
    \bar{R}_0(\lambda) \\
    \bar{R}_1(\lambda) \\
    \vdots \\
    \bar{R}_{k-1}(\lambda) \\
  \end{array}
\right)
+ \tau_k
\left(
  \begin{array}{c}
    0 \\
    0 \\
    \vdots \\
    \bar{R}_{k}(\lambda)
  \end{array}
\right).
\end{align}
which can be considered as derived from (\ref{eqn:reccurence-Tk-Rk}) by replacing $A$ by $\lambda$ and ignoring $v_0$.
In fact, by (\ref{eqn:lanczos-vk}), there holds
\begin{align}
Av_{k-1} = \sigma_{k-1} v_{k-1} + \tau_{k-1} v_{k-2} + \tau_k v_k.
\end{align}
Since $v_k = \bar{R}_k(A)v_0$, the above formula is converted to
\begin{align}
A\bar{R}_{k-1}(A)v_0 = \sigma_{k-1} \bar{R}_{k-1}(A)v_0 + \tau_{k-1} \bar{R}_{k-2}(A)v_0 + \tau_k \bar{R}_k(A) v_0.
\end{align}
Thus the polynomials $\bar{R}_k(\lambda)$ satisfy the reccurence
\begin{align}
\lambda\bar{R}_{k-1}(\lambda) = \sigma_{k-1} \bar{R}_{k-1}(\lambda) + \tau_{k-1} \bar{R}_{k-2}(\lambda) + \tau_k \bar{R}_k(\lambda).
\end{align}
Putting them in matrix-vector format, (\ref{eqn:eigenvalue-Tk}) is resulted.

From (\ref{eqn:eigenvalue-Tk}), it is easily seen that the roots $\bar{\lambda}_i$ of $\bar{R}_k(\lambda)$ is nothing but the eigenvalues of $T_k$, that is,
\begin{align}
\nonumber
T_k
\left(
  \begin{array}{c}
    \bar{R}_0(\bar{\lambda}_i) \\
    \bar{R}_1(\bar{\lambda}_i) \\
    \vdots \\
    \bar{R}_{k-1}(\bar{\lambda}_i) \\
  \end{array}
\right)
=
\bar{\lambda}_i
\left(
  \begin{array}{c}
    \bar{R}_0(\bar{\lambda}_i) \\
    \bar{R}_1(\bar{\lambda}_i) \\
    \vdots \\
    \bar{R}_{k-1}(\bar{\lambda}_i) \\
  \end{array}
\right).
\end{align}
The corresponding eigenvectors are formed by the function values of $\bar{R}_j$ at $\bar{\lambda}_i$, $j=0,\ldots,k-1$. Therefore, $\bar{R}_k(\lambda)$ or $R_k(\lambda)$ is the characteristic polynomial of $T_k$. If there is no roundoff error and the CG iteration process is  terminated at step $n$, $T_n$ is an $n\times n$ matrix similar to $A$ and $R_n(\lambda)$ is the characteristic polynomial of $A$. In this case, $\{\bar{R}_k(\lambda)\}$ or $\{R_k(\lambda)\}$ are the Sturm sequence of $T_n$.

Furthermore, if $v_m$ or $r_m$ vanishes for some $m<n$, $\bar{R}_m(\lambda)$ or $R_m(\lambda)$ has common factor with the characteristic polynomial of $A$. To see this, let $\varphi_i, i=1,\ldots,n$, be the normalized eigenvectors of the symmetric positive definite matrix $A$ and $\lambda_i$ the corresponding eigenvalues. Suppose that $r_0=\xi_{i_1} \varphi_{i_1}+\cdots+\xi_{i_l} \varphi_{i_l} \neq 0$, with $\xi_{i_1} \neq 0, \ldots, \xi_{i_l} \neq 0$. Then
\begin{align}
r_m = R_m(A)r_0 = R_m(\lambda_{i_1})\xi_{i_1} \varphi_{i_1} + \cdots + R_m(\lambda_{i_l})\xi_{i_l} \varphi_{i_l}.
\end{align}
Therefore $R_m(\lambda_{i_1})=0, \ldots, R_m(\lambda_{i_l})=0$. If the corresponding eigenvalues $\lambda_{i_1},\ldots,
\lambda_{i_l}$ are distinct, then $m=l$. Because if $m<l$, then the polynomial $R_m(\lambda)$ of degree $m<l$ will have more than $m$ different roots, a contradiction. If $m>l$, then the dimension of $\mbox{span}\{r_0,\ldots,r_{m-1}\}$ will not equal to the dimension of the Krylov subspace $\mbox{span}\{r_0,\ldots,A^{m-1}r_0\}$, again a contradiction. In this case of $m=l$, $R_m(\lambda)$ is a factor of the characteristic polynomial of $A$.

As for the roots of the conjugate polynomials $\bar{P}_k(\lambda)$, consider again the reccurence relation $AV_k = V_kT_k + \tau_k v_k e_k^T$ in Lanczos process. Due to the relation $\bar{W}_k=V_k L_k^{-T}$ between $\bar{p}_i$ and $v_i$ and the fact that $e_k^T L_k^{-T} = e_k^T$, there holds $A\bar{W}_k = \bar{W}_kL_k^TL_kD_k + \tau_k v_k e_k^T$. Note that $L_k^TL_kD_k$ is not symmetric. Scaling trick is applied,
\begin{align}
\label{eqn:recurrence-pk-bar}
A\bar{W}_kD_k^{-\frac{1}{2}} = \bar{W}_kD_k^{-\frac{1}{2}}D_k^{\frac{1}{2}}L_k^TL_kD_k^{\frac{1}{2}} + \tau_k v_k e_k^TD_k^{-\frac{1}{2}}.
\end{align}
Denote $\bar{\bar{W}}_k = \bar{W}_kD_k^{-\frac{1}{2}}$ and $\bar{T}_k = D_k^{\frac{1}{2}}L_k^TL_kD_k^{\frac{1}{2}}$. Since
$$
\bar{T}_k = D_k^{\frac{1}{2}}L_k^TL_kD_k^{\frac{1}{2}}D_k^{\frac{1}{2}}L_k^T L_k^{-T}D_k^{-\frac{1}{2}} = \left( D_k^{\frac{1}{2}}L_k^T \right) T_k \left( D_k^{\frac{1}{2}}L_k^T \right)^{-1},
$$
$\bar{T}_k$ is similar to $T_k$. By the relation $v_k = l_{k,k-1}\bar{p}_{k-1} + \bar{p}_k$ and the fact that
$\bar{p}_{k-1}/\sqrt{\delta_{k-1}}$ is the last column of $\bar{\bar{W}}_k$, (\ref{eqn:recurrence-pk-bar}) is rewritten as
\begin{align}
\nonumber
A\bar{\bar{W}}_k & = \bar{\bar{W}}_k\bar{T}_k + \frac{1}{\sqrt{\delta_{k-1}}}\tau_k l_{k,k-1} \bar{p}_{k-1} e_k^T + \frac{1}{\sqrt{\delta_{k-1}}}\tau_k \bar{p}_{k} e_k^T \\
\label{eqn:recurrence-pk-bar-bar}
& = \bar{\bar{W}}_k\bar{\bar{T}}_k + \frac{1}{\sqrt{\delta_{k-1}}}\tau_k \bar{p}_{k} e_k^T,
\end{align}
where
$$
\bar{\bar{T}}_k = \bar{T}_k +
\left(
  \begin{array}{ccc}
    & & \\
    & & \\
    & & \tau_k l_{k,k-1}
  \end{array}
\right)
=
\bar{T}_k +
\left(
  \begin{array}{ccc}
    & & \\
    & & \\
    & & \frac{\beta_{k-1}}{\alpha_{k-1}}
  \end{array}
\right),
$$
and $\alpha_k, \beta_k$ are quantities in CG iteration.
Proceed as the arguments for the roots of $\bar{R}_k(\lambda)$, it can be seen that the roots of the conjugate polynomial $P_k(\lambda)$ or $\bar{P}_k(\lambda)$, are the eigenvalues of $\bar{\bar{T}}_k$ which is a modification of $\bar{T}_k$, with $\bar{T}_k$ similar to $T_k$. For the roots of such sequences of polynomials of more general conjugate gradient method, see \cite{rootspolynomials-Manteuffel}.

\subsection{Duality between residual polynomials and $n$-dimensional geometry}
Recall that in the CG iteration process, the residual vectors are orthogonal with respect to the Euclidean inner product $(\cdot,\cdot)$, that is, $(r_i,r_j)=0, i\ne j$. Since the residual polynomial $R_i(\lambda)$ is associated with $r_i$, a natural question is that whether the polynomials $R_i(\lambda)$ are orthogonal in some sense? The answer is yes, provided in the original paper \cite{Hestenes-Stiefel}.

In order to explain the idea, as above let $\varphi_i, i=1,\ldots,n$ be the normalized eigenvectors of the symmetric positive definite matrix $A$ and $\lambda_i$ the corresponding eigenvalues. Suppose that $x_0$ is the initial iterate such that $r_0=\xi_1 \varphi_1+\ldots+\xi_n \varphi_n$ with all $\xi_i\ne 0$. In this setting, let us relate the inner product of $r_i, r_j$ with the polynomials $R_i(\lambda), R_j(\lambda)$ as follows,
\begin{align}
\nonumber
(r_i,r_j) & = \left( R_i(A)r_0, R_j(A)r_0 \right)\\
\nonumber
& = \left( R_i(\lambda_1)\xi_1\varphi_1+\ldots+R_i(\lambda_n)\xi_n\varphi_n, R_j(\lambda_1)\xi_1\varphi_1+\ldots+R_j(\lambda_n)\xi_n\varphi_n \right)\\
\label{eqn:duality-ri-Ri-1}
& = \xi_1^2 R_i(\lambda_1) R_j(\lambda_1)+\ldots+\xi_n^2 R_i(\lambda_n) R_j(\lambda_n).
\end{align}
The point is that whether (\ref{eqn:duality-ri-Ri-1}) can be viewed as an inner product of the polynomials $R_i(\lambda), R_j(\lambda)$. Define a step function $m(\lambda)$ as follows,
\begin{align}
\label{eqn:definition-mass-distribution}
m(\lambda)=
\left\{
\begin{aligned}
& 0,\quad \lambda<\lambda_1\\
& \xi_1^2, \quad \lambda_1 \le \lambda < \lambda_2\\
& \vdots\\
& \xi_1^2+\ldots+\xi_i^2,\quad \lambda_i \le \lambda < \lambda_{i+1}\\
& \xi_1^2+\ldots+\xi_n^2=1,\quad \lambda_n \le \lambda.
\end{aligned}
\right.
\end{align}
It is easily seen that $m(\lambda)$ is a nonnegative and nondecreasing function. The Riemann-Stieltjes integral exists for any continuous function $f(\lambda)$ with respect to $m(\lambda)$, and
\begin{align}
\int_0^c f(\lambda)dm(\lambda) = \xi_1^2 f(\lambda_1)+\ldots+\xi_n^2 f(\lambda_n),
\end{align}
where $c>\lambda_n$ is a constant. Under this definition,
\begin{align}
\label{eqn:duality-ri-Ri-2}
\int_0^c R_i(\lambda) R_j(\lambda)dm(\lambda) = (r_i,r_j).
\end{align}
That is the polynomials $R_i(\lambda), R_j(\lambda)$ are orthogonal with respect to this Riemann-Stieltjes integral. Note that $r_0,r_1,\ldots,r_{n-1}$ are orthogonal vectors of $\mathbb{R}^n$ and thus are a basis. Also note that $R_0(\lambda),R_1(\lambda),\ldots,R_{n-1}(\lambda)$ are a basis of the polynomial space $\mathbb{P}^{n-1}$. If the correspondence $\lambda^k \leftrightarrow A^kr_0$ is specified, then the $n$-dimensional space $\mathbb{R}^{n}$ is isomorphic to the polynomial space $\mathbb{P}^{n-1}$. In the original paper \cite{Hestenes-Stiefel}, Hestenes and Stiefel called the polynomials $R_i(\lambda)$ orthogonal polynomials based on such orthogonality.

As for the conjugate vectors and conjugate polynomials, there are similar relations,
\begin{align}
\nonumber
(Ap_i,p_j) & = \left( AP_i(A)r_0, P_j(A)r_0 \right)\\
\nonumber
& = \xi_1^2 \lambda_1 P_i(\lambda_1) P_j(\lambda_1)+\ldots+\xi_n^2 \lambda_n P_i(\lambda_n) P_j(\lambda_n)\\
\label{eqn:duality-pi-Pi}
& = \int_0^c \lambda P_i(\lambda) P_j(\lambda)dm(\lambda).
\end{align}
Since the conjugate vectors $p_i$ satisfy $(Ap_i,p_j)=0$, the corresponding conjugate polynomials $P_i(\lambda)$ are orthogonal with respect to the weight function $\lambda$.

\section{Convergence rate}
Since CG is closely related to the steepest descent method, the convergence rates of the two methods will be reviewed and compared. Note that given a vector $b$ and a symmetric positive definite matrix $A$, the iterative sequence $\{x_k\}$ of the two methods are completely dertimined by the initial guess $x_0$. As for two successive iterates, for the steepest descent, $x_{k+1}$ is determined totally by $x_k$; for CG, $x_{k+1}$ is determined by $x_k$ and the conjugate direction $p_{k}$. Usually there are two kinds of measurements for convergence rate of an iterative method. One is the ratio of every two successive error norms,
\begin{align}
\frac{\|x_{k+1}-x_*\|}{\|x_{k}-x_*\|},
\end{align}
which is usually adopted for the steepest descent, while the another is the total effect of the ratio in $k$ steps,
\begin{align}
\frac{\|x_{k}-x_*\|}{\|x_{0}-x_*\|}.
\end{align}
which is usually used for CG. Let us call the former two-term ratio and the latter $k$-term ratio.

Note that
\begin{align}
J(x)=\frac{1}{2}x^TAx-x^Tb=\frac{1}{2}x^TAx-x^TAx_*=\frac{1}{2}(x-x_*)^TA(x-x_*) - \frac{1}{2}x_*^TAx_*.
\end{align}
Thus the $A$-norm of error satisfies
\begin{align}
\label{eqn:error-A-norm-Jx}
(x-x_*)^TA(x-x_*)=2J(x)+x_*^TAx_*.
\end{align}
At $x_k$, if $ J(x)$ is marching along the direction $d_k$, by (\ref{eqn:descent-quantity}) and the fact that $\|x_{k}-x_*\|_A^2 = r_k^TA^{-1}r_k$, the $A$-norms of successive errors satisfy
\begin{align}
\label{eqn:successive-error-A-norm-SD}
\|x_{k+1}-x_*\|_A^2 = \|x_{k}-x_*\|_A^2 - \frac{\left(r_k^Td_k\right)^2}{d_k^TAd_k} = \|x_{k}-x_*\|_A^2 \left( 1 - \frac{\left(r_k^Td_k\right)^2}{d_k^TAd_k\cdot r_k^TA^{-1}r_k} \right).
\end{align}

\subsection{Convergence rate of steepest descent}
In this case, $d_k=r_k$. Define
\begin{align}
Q_{sd}(x_k)\triangleq \frac{\|x_{k+1}-x_*\|_A}{\|x_{k}-x_*\|_A},
\end{align}
A reasonable definition of the convergence factor for the steepest descent is
\begin{align}
Q_{sd}=\max_{x_k}Q_{sd}(x_k).
\end{align}
Note that by (\ref{eqn:error-A-norm-Jx}),
\begin{align}
\nonumber
\|x_{k+1}-x_*\|_A^2 & = 2J(x_{k+1})+x_*^TAx_*\\
\nonumber
& = \min_{\alpha}2J(x_k+\alpha r_k) + x_*^TAx_* \\
\nonumber
& = \min_{\alpha}(x_k+\alpha r_k-x_*)^TA(x_k+\alpha r_k-x_*).
\end{align}
Therefore, in essence, the convergence factor of the steepest descent is a max-min problem
\begin{align}
Q_{sd}^2=\max_{x_k}Q_{sd}^2(x_k) = \max_{x_k}\min_{\alpha}\frac{(x_k+\alpha r_k-x_*)^TA(x_k+\alpha r_k-x_*)}{(x_k-x_*)^TA(x_k-x_*)}.
\end{align}
As is seen in (\ref{eqn:successive-error-A-norm-SD}), the inner optimization problem min has a solution and
\begin{align}
\label{prob:convergence-factor-general}
Q_{sd}^2 = \max_{x_k}\left( 1 - \frac{\left(r_k^Td_k\right)^2}{d_k^TAd_k\cdot r_k^TA^{-1}r_k} \right).
\end{align}
With $d_k=r_k$, 
\begin{align}
\frac{\left(r_k^Td_k\right)^2}{d_k^TAd_k\cdot r_k^TA^{-1}r_k} = \frac{\left(r_k^Tr_k\right)^2}{r_k^TAr_k\cdot r_k^TA^{-1}r_k} =
\left( \frac{r_k^Tr_k}{r_k^TAr_k} \right)
\left( \frac{r_k^Tr_k}{r_k^TA^{-1}r_k} \right).
\end{align}
(\ref{prob:convergence-factor-general}) is reduced to
\begin{align}
\label{prob:convergence-factor-steepest-descent}
Q_{sd}^2 = \max_{r_k}\left( 1 - \left( \frac{r_k^Tr_k}{r_k^TAr_k} \right)
\left( \frac{r_k^Tr_k}{r_k^TA^{-1}r_k} \right) \right).
\end{align}
Setting $v = \frac{r_k}{\|r_k\|_2}$, problem (\ref{prob:convergence-factor-steepest-descent}) is related to the following constrained optimization problem,
\begin{align}
\label{prob:harmonic-optimization-1}
\left\{
\begin{aligned}
\max_{v} \left( v^TAv \right) \left( v^TA^{-1}v \right)\\
s.t.\quad \|v\|_2 = 1.
\end{aligned}
\right.
\end{align}
Since $A$ is symmetric positive definite, using the spectral information of $A$, problem (\ref{prob:harmonic-optimization-1}) is equivalent to
\begin{align}
\label{prob:harmonic-optimization-2}
\left\{
\begin{aligned}
\max_{\xi_i} \left( \sum_{i=1}\lambda_i \xi_i^2 \right)\left( \sum_{i=1}\lambda_i^{-1} \xi_i^2 \right)\\
s.t.\quad \sum_{i}\xi_i^2 = 1,
\end{aligned}
\right.
\end{align}
By setting $t_i = \xi_i^2$, it is reduced to
\begin{align}
\label{prob:harmonic-optimization-3}
\left\{
\begin{aligned}
\max_{t_i} \left( \sum_{i=1}\lambda_i t_i \right)\left( \sum_{i=1}\lambda_i^{-1} t_i \right)\\
s.t.\quad \sum_{i}t_i = 1\\
t_i \ge 0.
\end{aligned}
\right.
\end{align}
Problem (\ref{prob:harmonic-optimization-3}) has an explicit solution. Before the explicit solution is derived, a lemma is needed.

\begin{lemma}
\label{lem:harmonic-variation}
Let $0< \lambda_1 < \lambda_2 < \lambda_3$ be three positive numbers and denote $c_{ij} = \left( \sqrt{\frac{\lambda_j}{\lambda_i}} - \sqrt{\frac{\lambda_i}{\lambda_j}} \right)^2$. Then there holds
\begin{align}
\label{eqn:harmonic-variation-1}
\sqrt{c_{12}} + \sqrt{c_{23}} < \sqrt{c_{13}},
\end{align}
and therefore
\begin{align}
\label{eqn:harmonic-variation-2}
c_{12} + c_{23} < c_{13}.
\end{align}
\end{lemma}

\begin{proof}
It is easy to verify that
\begin{align}
\nonumber
\frac{\lambda_2-\lambda_1}{\sqrt{\lambda_1}} \frac{\lambda_3-\lambda_2}{\sqrt{\lambda_2 \lambda_3}\left( \sqrt{\lambda_2}+\sqrt{\lambda_3} \right)}
<
\frac{\lambda_3-\lambda_2}{\sqrt{\lambda_3}} \frac{\lambda_2-\lambda_1}{\sqrt{\lambda_1 \lambda_2}\left( \sqrt{\lambda_1}+\sqrt{\lambda_2} \right)}.
\end{align}
Thus
\begin{align}
\nonumber
\frac{\lambda_2-\lambda_1}{\sqrt{\lambda_1}} \left( \frac{1}{\sqrt{\lambda_2}}-\frac{1}{\sqrt{\lambda_3}} \right)
<
\frac{\lambda_3-\lambda_2}{\sqrt{\lambda_3}} \left( \frac{1}{\sqrt{\lambda_1}}-\frac{1}{\sqrt{\lambda_2}} \right).
\end{align}
Moving terms to both sides gives
\begin{align}
\nonumber
\frac{\lambda_2-\lambda_1}{\sqrt{\lambda_1 \lambda_2}} + \frac{\lambda_3-\lambda_2}{\sqrt{\lambda_2 \lambda_3}}
<
\frac{\lambda_2-\lambda_1}{\sqrt{\lambda_1 \lambda_3}} + \frac{\lambda_3-\lambda_2}{\sqrt{\lambda_1 \lambda_3}}
=
\frac{\lambda_3-\lambda_1}{\sqrt{\lambda_1 \lambda_3}},
\end{align}
which is nothing but (\ref{eqn:harmonic-variation-1}).
\end{proof}

\begin{proposition}
\label{proposi:harmonic-optimization}
Let $0<\lambda_1 \le \lambda_2 \le \ldots \le \lambda_n$ be the eigenvalues of the  symmetric positive definite matrix $A$. Then the maximum value of problem (\ref{prob:harmonic-optimization-3}) is
\begin{align}
\frac{1}{4}\left( \lambda_1+\lambda_n \right)\left( \frac{1}{\lambda_1}+\frac{1}{\lambda_n} \right).
\end{align}
\end{proposition}

\begin{proof}
Rewrite the product as
\begin{align}
\nonumber
& f(t_1,\ldots,t_n)\\
\nonumber
\triangleq & \left( \lambda_1 t_1+\ldots+\lambda_n t_n \right)\left( \lambda_1^{-1} t_1+\ldots+\lambda_n^{-1} t_n \right)\\
\nonumber
= & t_1^2+\ldots+t_n^2+\frac{\lambda_1}{\lambda_2}t_1 t_2+\ldots+\frac{\lambda_1}{\lambda_n}t_1 t_n+\frac{\lambda_2}{\lambda_1}t_2 t_1+\ldots+\frac{\lambda_2}{\lambda_n}t_2 t_n + \ldots\\
\nonumber
& \quad +\frac{\lambda_n}{\lambda_1}t_n t_1+\ldots+\frac{\lambda_n}{\lambda_{n-1}}t_{n} t_{n-1}\\
\nonumber
= & \left( t_1+t_2+\ldots+t_n \right)^2 + \left( \frac{\lambda_1}{\lambda_2} + \frac{\lambda_2}{\lambda_1} - 2 \right)t_1 t_2 +\ldots +
\left( \frac{\lambda_1}{\lambda_n} + \frac{\lambda_n}{\lambda_1} - 2 \right)t_1 t_n\\
\nonumber
& \quad + \left( \frac{\lambda_2}{\lambda_3} + \frac{\lambda_3}{\lambda_2} - 2 \right)t_2 t_3 +\ldots +
\left( \frac{\lambda_2}{\lambda_n} + \frac{\lambda_n}{\lambda_2} - 2 \right)t_2 t_n\\
\nonumber
& \quad + \ldots\\
\nonumber
& \quad + \left( \frac{\lambda_{n-1}}{\lambda_n} + \frac{\lambda_n}{\lambda_{n-1}} - 2 \right)t_{n-1} t_n\\
\nonumber
= & 1 + \left( \sqrt{\frac{\lambda_2}{\lambda_1}} - \sqrt{\frac{\lambda_2}{\lambda_1}} \right)^2 t_1 t_2 +\ldots +
\left( \sqrt{\frac{\lambda_n}{\lambda_1}} - \sqrt{\frac{\lambda_1}{\lambda_n}} \right)^2 t_1 t_n\\
\nonumber
& \quad + \left( \sqrt{\frac{\lambda_3}{\lambda_2}} - \sqrt{\frac{\lambda_2}{\lambda_3}} \right)^2 t_2 t_3 +\ldots +
\left( \sqrt{\frac{\lambda_n}{\lambda_2}} - \sqrt{\frac{\lambda_2}{\lambda_n}} \right)^2 t_2 t_n\\
\nonumber
& \quad + \ldots\\
\nonumber
& \quad + \left( \sqrt{\frac{\lambda_n}{\lambda_{n-1}}} - \sqrt{\frac{\lambda_{n-1}}{\lambda_n}} \right)^2 t_{n-1} t_n\\
\nonumber
= & 1 + c_{12}t_1 t_2 + \ldots + c_{n-1,n}t_{n-1} t_n\\
\label{eqn:harmonic-optimization-4}
= & 1 + \frac{1}{2}t^T C t,
\end{align}
where $t = (t_1,\ldots,t_n)^T$, and
\begin{equation}
C = \left(
  \begin{array}{ccccc}
    0 & c_{12} & c_{13} & \ldots & c_{1n} \\
    c_{12} & 0 & c_{23} & \ldots & c_{2n} \\
    c_{13} & \ddots & 0 & \ddots & \vdots \\
    \vdots & & \ddots & 0 & c_{n-1,n}\\
    c_{1n} & c_{2n} & \ldots & c_{n-1,n} & 0\\
  \end{array}
\right), \quad
c_{ij} = \left( \sqrt{\frac{\lambda_j}{\lambda_i}} - \sqrt{\frac{\lambda_i}{\lambda_j}} \right)^2. \nonumber
\end{equation}
Therefore (\ref{prob:harmonic-optimization-3}) is reduced further to the following quadratic programming
\begin{align}
\label{prob:harmonic-optimization-4}
\left\{
\begin{aligned}
\max_{t} & f(t) = 1 + \frac{1}{2}t^T C t\\
\mbox{s.t.} & \quad g(t) = e^T t - 1 = 0\\
 & \quad t_i \ge 0,
\end{aligned}
\right.
\end{align}
where $e = (1,\ldots,1)^T$.

If $n=2$, $f(t_1,t_2) = 1 + c_{12}t_1 t_2 = 1 + c_{12}t_1 (1-t_1)$ and the maximum is attained when $t_1 = t_2 = \frac{1}{2}$. Consider the case $n \ge 3$ in the sequel.

Assume that $0<\lambda_1 < \lambda_2 < \ldots < \lambda_n$. Since $\lambda_i \ne \lambda_j$, $c_{ij}>0$. Let $(\bar{t}_1,\ldots,\bar{t}_n)$ be a maximum point. It is to be shown that $\bar{t}_i=0$, $i=2,\ldots,n-1$. By contradition, suppose that $\bar{t}_2 > 0$. Construct a marching direction
\begin{align}
d = \left(\eta_1,-(\eta_1 + \eta_n),0,\ldots,0,\eta_n \right)^T,
\end{align}
where $\eta_1 = \frac{c_{2n}}{c_{1n}} > 0$ and $\eta_n = \frac{c_{12}}{c_{1n}} > 0$. Since $\bar{t}_2 > 0$, $t = \bar{t} + \alpha d$ will be a feasible point if $\alpha > 0$ is sufficiently small. For this direction, there holds $Cd > 0$ component-wise. In fact,
\begin{equation}
\left(
  \begin{array}{ccccc}
    0 & c_{12} & c_{13} & \ldots & c_{1n} \\
    c_{12} & 0 & c_{23} & \ldots & c_{2n} \\
    c_{13} & \ddots & 0 & \ddots & \vdots \\
    \vdots & & \ddots & 0 & c_{n-1,n}\\
    c_{1n} & c_{2n} & \ldots & c_{n-1,n} & 0\\
  \end{array}
\right)
\left(
  \begin{array}{c}
    \eta_1 \\
    -(\eta_1+\eta_n) \\
    0 \\
    \vdots\\
    0 \\
    \eta_n \\
  \end{array}
\right)
=
\left(
  \begin{array}{c}
    c_{1n}\eta_n - c_{12}(\eta_1+\eta_n) \\
    c_{12}\eta_1 + c_{2n}\eta_n \\
    c_{13}\eta_1 - c_{23}(\eta_1+\eta_n) + c_{3n}\eta_n \\
    \vdots\\
    c_{1,n-1}\eta_1 - c_{2,n-1}(\eta_1+\eta_n) + c_{n-1,n}\eta_n \\
    c_{1n}\eta_1 - c_{2n}(\eta_1+\eta_n) \\
  \end{array}
\right). \nonumber
\end{equation}
By Lemma \ref{lem:harmonic-variation},
\begin{align}
\nonumber
\eta_1 + \eta_n = \frac{c_{2n}}{c_{1n}} + \frac{c_{12}}{c_{1n}} < 1.
\end{align}
Therefore, for the first component
\begin{align}
\nonumber
c_{1n}\eta_n - c_{12}(\eta_1+\eta_n) = c_{12} - c_{12}(\eta_1+\eta_n) > 0.
\end{align}
Similarly for the last component
\begin{align}
\nonumber
c_{1n}\eta_1 - c_{2n}(\eta_1+\eta_n) > 0.
\end{align}
For the $i$-th component, $i=3,\ldots,n-1$, by Lemma \ref{lem:harmonic-variation} again,
\begin{align}
\nonumber
& c_{1i}\eta_1 - c_{2i}(\eta_1+\eta_n) + c_{in}\eta_n\\
\nonumber
= & \frac{1}{c_{1n}} \left( c_{1i}c_{2n} - c_{2i}(c_{2n}+c_{12}) + c_{in}c_{12} \right)\\
\nonumber
= & \frac{1}{c_{1n}} \left( (c_{1i} - c_{2i})c_{2n} - c_{2i}c_{12} + c_{in}c_{12} \right)\\
\nonumber
> & \frac{1}{c_{1n}} \left( c_{12}c_{2n} - c_{2i}c_{12} + c_{in}c_{12} \right)\\
\nonumber
> & \frac{1}{c_{1n}} \left( c_{in}c_{12} + c_{in}c_{12} \right)\\
\nonumber
> & 0.
\end{align}
Thus $Cd>0$ component-wise. As a result,
\begin{align}
\label{eqn:positive-direction-derivative}
\bar{t}^T Cd \ge \bar{t}_2 \left( c_{12}\eta_1 + c_{2n}\eta_n \right) > 0.
\end{align}
Along this direction $d$, if $\alpha > 0$ is sufficiently small, $t = \bar{t} + \alpha d$ is feasible and
\begin{align}
\nonumber
f(\bar{t}+\alpha d) & = 1 + \frac{1}{2}\left( \bar{t}+\alpha d \right)^T C \left( \bar{t}+\alpha d \right)\\
\nonumber
& = 1 + \frac{1}{2}\left( \bar{t}^T C t + 2\bar{t}^T Cd \alpha + d^TCd \alpha^2 \right)\\
\label{eqn:contradiction}
& > f(\bar{t}),
\end{align}
which is a contradition to the assumption that $\bar{t}$ is a maximum point. By similar arguments, it can be proved that $\bar{t}_i = 0$, $i = 3,\ldots,n-1$. Therefore $f(\bar{t}) = 1 + c_{1n} \bar{t}_1 \bar{t}_n$. Analogue to the case $n=2$, the maximum is reached when $\bar{t}_1 = \bar{t}_n = \frac{1}{2}$, and the maximum is just
\begin{align}
\nonumber
& \left( \frac{1}{2}\lambda_1 + 0\cdot \lambda_2 + \cdots + 0\cdot \lambda_{n-1} + \frac{1}{2}\lambda_n\right)
\left( \frac{1}{2}\frac{1}{\lambda_1} + 0\cdot \frac{1}{\lambda_2} + \cdots + 0\cdot \frac{1}{\lambda_{n-1}} + \frac{1}{2}\frac{1}{\lambda_n} \right)\\
= & \frac{1}{4}\left( \lambda_1+\lambda_n \right)\left( \frac{1}{\lambda_1}+\frac{1}{\lambda_n} \right).
\end{align}
Converted to the original problem (\ref{prob:harmonic-optimization-1}), $\bar{t}_1 = \bar{t}_n = \frac{1}{2}$ corresponds to
\begin{align}
v = \frac{1}{\sqrt{2}}\varphi_1 + \frac{1}{\sqrt{2}}\varphi_n.
\end{align}

If there are repeated eigenvalues, suppose the distinct eigenvalues are listed as $\tilde{\lambda}_1 < \ldots < \tilde{\lambda}_k$. The $t_i$'s can be separated into groups corresponding to distinct eigenvalues. For example, if $\tilde{\lambda}_1=\lambda_1=\lambda_2<\tilde{\lambda}_2=\lambda_3=\lambda_4<\ldots$. Then $t_1,t_2$ can be combined together as a new variable $\tilde{t}_1=t_1+t_2$ and $t_3,t_4$ combined as $\tilde{t}_2=t_3+t_4$. And the above argument applies.
\end{proof}

By Proposition \ref{proposi:harmonic-optimization}, the convergence factor of the steepest descent is
\begin{align}
Q_{sd} = \sqrt{1-\frac{4\lambda_1 \lambda_n}{\left( \lambda_1+\lambda_n \right)^2}}=\frac{\lambda_n - \lambda_1}{\lambda_n + \lambda_1},
\end{align}
which is the same as the convergence factor derived by Chebyshev polynomial in textbook. In fact, in textbook, the following factor $\bar{Q}_{sd}$ is taken as the convergence factor,
\begin{align}
\bar{Q}_{sd}^2 = \min_{\alpha}\max_{x_k}\frac{(x_k+\alpha r_k-x_*)^TA(x_k+\alpha r_k-x_*)}{(x_k-x_*)^TA(x_k-x_*)}.
\end{align}
Note that for a function of two arguments $S(\alpha,x)$,
\begin{align}
\max_{x}\min_{\alpha}S(\alpha,x) \le \min_{\alpha}\max_{x}S(\alpha,x).
\end{align}
Therefore $\bar{Q}_{sd}$ is an upper bound of $Q_{sd}$, i.e., $Q_{sd} \le \bar{Q}_{sd}$. It turns out that the minimum value $\bar{Q}_{sd}$ is also $\frac{\lambda_n - \lambda_1}{\lambda_n + \lambda_1}$, that is $Q_{sd} = \bar{Q}_{sd}$.

\subsection{Convergence rate of CG}
For CG, the marching direction is the conjugate direction, i.e. $d_k = p_k$, and the successive error $A$-norms satisfy
\begin{align}
\label{eqn:successive-error-A-norm-CG}
\|x_{k+1}-x_*\|_A^2 = \|x_{k}-x_*\|_A^2 - \frac{\left(r_k^Tp_k\right)^2}{p_k^TAp_k} = \|x_{k}-x_*\|_A^2 \left( 1 - \frac{\left(r_k^Tp_k\right)^2}{p_k^TAp_k\cdot r_k^TA^{-1}r_k} \right).
\end{align}
Note that for CG, $x_{k+1}$ is not determined by $x_k$ only. Rather, it is determined by both $x_k$ and $p_k$, and the convergence factor depends on $r_k$ and $p_k$. Therefore define
\begin{align}
Q_{cg}(r_k,p_k)\triangleq \frac{\|x_{k+1}-x_*\|_A}{\|x_{k}-x_*\|_A}.
\end{align}
The convergence factor of two-term ratio for CG is defined as
\begin{align}
Q_{cg}=\sup_{r_k,p_k}Q_{cg}(r_k,p_k).
\end{align}
By (\ref{eqn:successive-error-A-norm-CG}), the following quantity should be maximized,
\begin{align}
\frac{p_k^TAp_k\cdot r_k^TA^{-1}r_k}{\left(r_k^Tp_k\right)^2}.
\end{align}
By the properties of the iterates of CG, $r_k^Tp_k = r_k^Tr_k$, and $(p_k-r_k)^TAp_k = 0$. Therefore
\begin{align}
\nonumber
p_k^TAp_k & = r_k^TAr_k+2(p_k-r_k)^TAr_k+(p_k-r_k)^TA(p_k-r_k)\\
\nonumber
& = r_k^TAr_k-2(p_k-r_k)^TA(p_k-r_k)+(p_k-r_k)^TA(p_k-r_k)\\
\nonumber
& = r_k^TAr_k-(p_k-r_k)^TA(p_k-r_k).
\end{align}
Replacing $r_k^Tp_k$ by $r_k^Tr_k$, gives
\begin{align}
\nonumber
\frac{p_k^TAp_k\cdot r_k^TA^{-1}r_k}{\left(r_k^Tp_k\right)^2} & =
\left(\frac{r_k^TAr_k-(p_k-r_k)^TA(p_k-r_k)}{r_k^Tr_k}\right)\left(\frac{r_k^TA^{-1}r_k}{r_k^Tr_k}\right)\\
\nonumber
& \le \left(\frac{r_k^TAr_k}{r_k^Tr_k}\right)\left(\frac{r_k^TA^{-1}r_k}{r_k^Tr_k}\right)\\
\nonumber
& \le \frac{1}{4}\left( \lambda_1+\lambda_n \right)\left( \frac{1}{\lambda_1}+\frac{1}{\lambda_n} \right).
\end{align}
Note that the maximum value of $Q_{cg}(r_k,p_k)$ may not exist, but $(p_k-r_k)^TA(p_k-r_k)$ may be arbitrarily small. So
\begin{align}
Q_{cg}=\sup_{r_k,p_k}Q_{cg}(r_k,p_k)=\sqrt{1-\frac{4\lambda_1 \lambda_n}{\left( \lambda_1+\lambda_n \right)^2}}=\frac{\lambda_n - \lambda_1}{\lambda_n + \lambda_1}.
\end{align}
In this sense the convergence factor $Q_{cg}$ of two-term ratio for CG is the same as $Q_{sd}$ for the steepest descent.

At first glance, this convergence factor $Q_{cg}$ may seem too large for CG, because in textbook, the convergence estimate is
\begin{align}
\label{eqn:convergence-estimate-CG}
\|x_{k}-x_*\|_A \le 2\left( \frac{\sqrt{\lambda_n}-\sqrt{\lambda_1}}{\sqrt{\lambda_n}+\sqrt{\lambda_1}} \right)^k \|x_{0}-x_*\|_A.
\end{align}
However, note that the quantity
$\frac{\sqrt{\lambda_n}-\sqrt{\lambda_1}}{\sqrt{\lambda_n}+\sqrt{\lambda_1}}$
should be considered as the average convergence factor of the $k$-term ratio, and indivisual two-term ratio may exceed this average convergence factor, as the numerical experiment shows, see Table \ref{tab1:ratio}. Here the matrix $A$ is randomly chosen. 'ratio-2' represents $\|x_{k}-x_*\|_A/\|x_{k-1}-x_*\|_A$ and 'ratio-k' represents $\sqrt[k]{\|x_{k}-x_*\|_A/\|x_{0}-x_*\|_A}$. For more theories on the convergence rate of CG, see \cite{convergencerate-VanderSluis,convergencerate-Sleijpen,convergencerate-Greenbaum}.

\begin{table}[!h]
\renewcommand{\arraystretch}{1.10}
\tabcolsep 0pt \caption{Two-term ratios and k-term mean ratios.}
\vspace*{-12pt} \label{tab1:ratio}
\begin{center}
\def\temptablewidth{0.8\textwidth}
{\rule{\temptablewidth}{1pt}}
\begin{tabular*}{\temptablewidth}{@{\extracolsep{\fill}}cccccccc}
$k$ & $\frac{\lambda_n-\lambda_1}{\lambda_n+\lambda_1}$ & ratio-2 & ratio-k & $\frac{\sqrt{\lambda_n}-\sqrt{\lambda_1}}{\sqrt{\lambda_n}+\sqrt{\lambda_1}}$ &
$k$ & ratio-2 & ratio-k \\
\cline{1-5}\cline{6-8}
& & & & & & &\\
1 &   &   0.996 &  0.996 &   & 14  &   0.964 &  0.958\\
2 &   &   0.998 &  0.997 &   & 15  &   0.906 &  0.954\\
3 &   &   0.996 &  0.997 &   & 16  &   0.747 &  0.940\\
4 &   &   0.989 &  0.995 &   & 17  &   0.810 &  0.932\\
5 &   &   0.992 &  0.994 &   & 18  &   0.746 &  0.920\\
6 &   &   0.989 &  0.993 &   & 19  &   0.854 &  0.917\\
7 &  0.999 &   0.985 &  0.992 &  0.969 & 20  &   0.996 &  0.921\\
8 &   &   0.964 &  0.989 &   & 21  &   0.001 &  0.677\\
9 &   &   0.946 &  0.984 &   & 22  &   0.000 &  0.474\\
10 &   &   0.908 &  0.976 &   & 23  &   0.100 &  0.443\\
11 &   &   0.940 &  0.973 &   & 24  &   0.028 &  0.395\\
12 &   &   0.845 &  0.961 &   & 25  &   0.055 &  0.365\\
13 &   &   0.913 &  0.958 &
\end{tabular*}
{\rule{\temptablewidth}{1pt}}
\end{center}
\end{table}


\section{CG in Hilbert space}
One of the origins of equation (\ref{eqn:Axb}) is discretization of second order self-adjoint differential equation, for example, the Poisson equation with homogeneous Dirichlet boundary condition,
\begin{align}
\label{eqn:poisson-problem}
\left\{
\begin{aligned}
-\Delta u + cu= f,\quad \Omega\\
u = 0,\quad \partial\Omega.
\end{aligned}
\right.
\end{align}
A natural question is that whether the CG method can be applied to the self-adjoint operator equation directly, before its discretization. The answer is yes, see \cite{CG-hilbert-Hayes,CG-hilbert-Herzog-Sachs}. But the function spaces and operators involved should be chosen thoughtfully. Problem (\ref{eqn:poisson-problem}) is taken to demonstrate the idea of construction of the CG method in function space, and to illustrate the relationship between the finite dimensial CG and the infinite dimensional CG.

The above elliptic problem can be rewritten as an operator equation in certain sense
\begin{align}
\mathscr{A} u = f.
\end{align}
Recall that in the CG algorithm, operations such as $Ax$ and $A^2x$ will be implicitly involved. If the CG is generalized to the operator equation (\ref{eqn:poisson-operator}), analogously $\mathscr{A}u$ and $\mathscr{A}^2u$ should be well defined. Therefore the domain space and the range space of $\mathscr{A}$ should be the same. If $\mathscr{A}$ is chosen as $-\Delta + c$, and the domain space of $\mathscr{A}$ chosen as $C^2_0(\Omega)$, $\mathscr{A}u$ is well defined but $\mathscr{A}^2u$ may not be well defined. Furthermore, $\mathscr{A}u$ may not lie in $C^2_0(\Omega)$. If the domain space of $\mathscr{A}$ is chosen as $L^2(\Omega)$, $\mathscr{A}$ is not defined in the whole domain space.

In PDE community, problem (\ref{eqn:poisson-problem}) usually is understood in weak sense, that is, find $u\in H_0^1(\Omega)$, such that
\begin{align}
\label{eqn:poisson-weak-sense}
\int_{\Omega}\nabla u \nabla v dx + c\int_{\Omega}u v dx = \int_{\Omega}fvdx,\quad \forall v\in H_0^1(\Omega).
\end{align}
In this sense, problem (\ref{eqn:poisson-problem}) can be firstly considered as
\begin{align}
\label{eqn:poisson-weak-operator}
\mathscr{L} u = f,\quad \mathscr{L}:X\to X^*,
\end{align}
where $X=H_0^1(\Omega)$, with inner product
\begin{align}
(u,v)_{X}=\int_{\Omega}\nabla u \nabla v dx + \int_{\Omega}u v dx.
\end{align}
In order to generalize the CG method, the image of the operator $\mathscr{L}$ should be pulled back to the domain space $X$. Rietz isomorphism  can do this job. Then, problem (\ref{eqn:poisson-problem}) is rewritten as the following operator equation
\begin{align}
\label{eqn:poisson-operator}
\mathscr{A}u\triangleq\mathscr{RL} u = \mathscr{R}f,
\end{align}
where $\mathscr{R}$ is the Rietz isomorphism $\mathscr{R}:X^*\to X$, i.e., $(\mathscr{R}f,v)_{X}=\langle f, v \rangle_{X^*,X}$.

On the other point of view, the solution $u$ of problem (\ref{eqn:poisson-weak-sense}) can be considered as the minimizer of the following quadratic functional
\begin{align}
\mathscr{J}(u)=\frac{1}{2}\langle \mathscr{L} u, u \rangle_{X^*,X} - \langle f, u \rangle_{X^*,X},\quad \mathscr{J}:X\to \mathbb{R}
\end{align}
and $\mathscr{L} u - f = 0$ corresponds to $\mathscr{J}'(u)=0$, similar to that $Ax - b = 0$ corresponds to $J'(u)=0$ in Section \ref{set:connections}. Here $\mathscr{J}'(u)\in X^*$ is the Fr\'{e}chlet derivative of $\mathscr{J}(u)$. Consider the steepest descend direction of $\mathscr{J}$ at $u$. Note the steepest descend direction of a function or functional is the minus gradient of that function or functional. It is important to distinguish the two notions of Fr\'{e}chlet derivative and gradient of a functional. For this example, the Fr\'{e}chlet derivative of $\mathscr{J}$ at $u$ is $\mathscr{J}'(u)\in X^*$ and the gradient is $\nabla\mathscr{J}(u)\in X$. They are related by $\nabla\mathscr{J}(u)=\mathscr{R}\mathscr{J}'(u)$. In finite dimensional Euclidian space $\mathbb{R}^n$, $\nabla J(u)=(J'(u))^T$, that is if $\mathbb{R}^n$ is considered as a column vector space, the Fr\'{e}chlet derivative of $J$ at $u$ is identified as a row vector, the gradient of $J$ is a column vector and the Riesz isomorphism is the operation of transpose.

After clarifying the setting and the notions, the CG method for $\mathscr{L} u - f = 0$ in function space can be derived. Suppose that $u_0$ is an initial guess. The residual $r_0=f-\mathscr{L} u_0=-\mathscr{J}'(u)\in X^*$. The first search direction is taken as the steepest descend direction $p_0=-\nabla \mathscr{J}'(u_0)=\mathscr{R}r_0\in X$. Note that $\langle \mathscr{L}u, p \rangle_{X^*,X}=\langle \mathscr{L}p, u \rangle_{X^*,X}$. Minimizing $\mathscr{J}(u_0+\alpha p_0)$ with respect to $\alpha$ gives
\begin{align}
\alpha_0 = \frac{\langle f-\mathscr{L}u_0, p_0 \rangle_{X^*,X}}{\langle \mathscr{L}p_0, p_0 \rangle_{X^*,X}} = \frac{\langle r_0, p_0 \rangle_{X^*,X}}{\langle \mathscr{L}p_0, p_0 \rangle_{X^*,X}}.
\end{align}
In passing,
\begin{align}
u_1 = u_0 + \alpha_0 p_0,\\
r_1 = r_0 - \alpha_0 \mathscr{L}p_0.
\end{align}
Note that $r_1\in X^*$. With $r_1$ and $p_0$ at hand, a direction $p_1$ conjugate to $p_0$ is constructed as
\begin{align}
p_1 = \mathscr{R}r_1+\beta_0 p_0,\quad \beta_0=-\frac{\langle \mathscr{L}p_0, \mathscr{R}r_1 \rangle_{X^*,X}}{\langle \mathscr{L}p_0, p_0 \rangle_{X^*,X}}.
\end{align}
Repeating these steps, the CG method in function space is resulted
\begin{align}
\label{eqn:cg-infinite-pk}
p_k & = \mathscr{R}r_k+\beta_{k-1} p_{k-1},\quad \beta_{k-1}=-\frac{\langle \mathscr{L}p_{k-1}, \mathscr{R}r_k \rangle_{X^*,X}}{\langle \mathscr{L}p_{k-1}, p_{k-1} \rangle_{X^*,X}},\\
\label{eqn:cg-infinite-uk}
u_{k+1} & = u_k + \alpha_k p_k,\quad \alpha_k = \frac{\langle r_k, p_k \rangle_{X^*,X}}{\langle \mathscr{L}p_k, p_k \rangle_{X^*,X}},\\
\label{eqn:cg-infinite-rk}
r_{k+1} & = r_k - \alpha_k \mathscr{L}p_k.
\end{align}

Another natural question arises. If (\ref{eqn:poisson-weak-operator}) is discretized with discretized algebraic equations $AU = F$, and the finite dimensional CG is applied to $AU = F$, then is there any connection between the infinite dimensional CG iterates for (\ref{eqn:poisson-weak-operator}) and the finite dimensional CG iterates for $AU = F$?

To be specific, consider the finite element discretization for (\ref{eqn:poisson-weak-sense}): Find $u_h = \sum\limits_{j=1}^{n}\eta_j\varphi_j\in X_h$, s.t.,
\begin{align}
\label{eqn:poisson-fem}
\int_{\Omega}\nabla u_h \nabla v_h dx + c\int_{\Omega}u_h v_h dx = \int_{\Omega}fv_h dx,\quad \forall v_h\in X_h,
\end{align}
or by short hand notation,
\begin{align}
a( u_h,v_h ) = (f,v_h),\quad \forall v_h\in X_h,
\end{align}
which can be considered as a discretized analog of the operator equation $\mathscr{L} u = f$ in (\ref{eqn:poisson-weak-operator})
\begin{align}
\label{eqn:poisson-weak-operator-fem}
\mathscr{L}_h u_h = f_h,\quad \mathscr{L}_h:X_h\to X_h^*,
\end{align}
where $f_h$ is the projection of $f$ in $X_h^*$. Therefore (\ref{eqn:poisson-fem}) can be rewritten as
\begin{align}
\langle \mathscr{L}_h u_h, v_h \rangle_{X_h^*,X_h} = \langle f_h, v_h \rangle_{X_h^*,X_h},\quad \mathscr{L}_h:X_h\to X_h^*,
\end{align}
Here $X_h$ is a finite element subspace of $X$ spaned by piece-wise linear basis functions $\varphi_1,\ldots,\varphi_n$, with inner product inherited from $X$,
\begin{align}
\left( u_h,v_h \right)_{X_h} = \int_{\Omega}\nabla u_h \nabla v_h dx + \int_{\Omega}u_h v_h dx.
\end{align}
In the following, $\langle \cdot, \cdot \rangle = \langle \cdot, \cdot \rangle_{X_h^*,X_h}$ will denote the dual pair and $\left( \cdot, \cdot \right)_{X_h}$ will denote the inner product.

It is worthwhile to note that usually it is not (\ref{eqn:poisson-weak-operator-fem}) that is solved by CG, rather, it is the representation of (\ref{eqn:poisson-weak-operator-fem}) that is solved by CG. Such representation can be derived by choosing $v_h=\varphi_j$ in (\ref{eqn:poisson-fem}), and is denoted as $AU = F$. Here $U=(\eta_1,\ldots,\eta_n)^T$ is the coefficients of the finite element basis functions $\varphi_i$, the elements $a_{ij}$ of the matrix $A$ and elments $F_i$ of $F$ are
\begin{align}
\label{eqn:aij}
a_{ij} & = \int_{\Omega}\nabla \varphi_j \nabla \varphi_i dx + c\int_{\Omega}\varphi_j \varphi_i dx,\\
F_i & =  \int_{\Omega}f\varphi_i dx.
\end{align}
In the finite element method, usually $A$ is formulated as $A=K+cM$, where $K$ is the stiff matrix and $M$ is the mass matrix. In this example, the matrix $A$ is symmetric positive definite. Define $Q_1:\mathbb{R}^n\to X_h$ and $Q_2:X_h^*\to \mathbb{R}^n$ as,
\begin{align}
Q_1
\left(
  \begin{array}{c}
    \eta_1 \\
    \eta_2 \\
    \vdots \\
    \eta_n
  \end{array}
\right)
= \sum_{j=1}^{n}\eta_i \varphi_i,
\quad
Q_2 r_h
=
\left(
  \begin{array}{c}
    \langle r_h,\varphi_1 \rangle \\
    \langle r_h,\varphi_2 \rangle \\
    \vdots \\
    \langle r_h,\varphi_n \rangle
  \end{array}
\right).
\end{align}
Note that $Q_2 = Q_1^*$ is the Hilbert conjugate of $Q_1$, i.e.,
\begin{align}
\langle r_h,Q_1 U \rangle = \langle Q_1^*r_h,U \rangle_{\mathbb{R}^n} = \langle Q_2r_h,U \rangle_{\mathbb{R}^n},\quad \forall U\in \mathbb{R}^n, r_h\in X_h^*.
\end{align}

In the language of mapping, with the base $\varphi_1,\ldots,\varphi_n$, the representation of $\mathscr{L}_h$ is $A=Q_2\mathscr{L}_h Q_1=Q_1^*\mathscr{L}_h Q_1$, with $Q_1U=u_h$ and $Q_2f_h=F$. That is the following diagram commutes.
\begin{align}
\xymatrix{
X_h \ar[rr]^{\mathscr{L}_h} & & X_h^* \ar[d]^{Q_1^*} \\
\mathbb{R}^n\ar[u]^{Q_1}\ar[rr]^{A} & & \mathbb{R}^n}
\end{align}
Let $w_h = \sum\limits_{j=1}^{n}w_i\varphi_i$ and $v_h = \sum\limits_{j=1}^{n}v_i\varphi_i$, then $\left( \mathscr{L}_h w_h,v_h \right)_{X_h}$ can be represented as
\begin{align}
w^T(K+cM)v = w^TAv,
\end{align}
where $w = (w_1,\ldots,w_n)^T,v = (v_1,\ldots,v_n)^T$.
In addition, the discrete analog $\mathscr{R}_h$ of the Riesz isomorphism $\mathscr{R}$ should be introduced, $\mathscr{R}_h:X_h^*\to X_h$. Let $r_h\in X_h^*$ and $w_h=\mathscr{R}_h r_h \in X_h$. The representation of $\mathscr{R}_h$ can be derived as follows. Suppose $w_h = \sum\limits_{j=1}^{n}w_i\varphi_i$ and denote $r_i = \langle r_h,\varphi_i \rangle$. By the definition of Riesz isomorphism,
\begin{align}
r_i = \langle r_h,\varphi_i \rangle = \left( \mathscr{R}_hr_h,\varphi_i \right)_{X_h} = \left( w_h,\varphi_i \right)_{X_h} = \sum_{j=1}^{n}\left( \varphi_j,\varphi_i \right)_{X_h}w_j.
\end{align}
In matrix-vector format,
$$
r = (K+M)w.
$$
Thus
$$
w = Rr,
$$
where $R=(K+M)^{-1}$. So $R$ is a representation of $\mathscr{R}_h$. Note that if $v_h = \sum\limits_{j=1}^{n}v_i\varphi_i$, then the inner product $\left( w_h,v_h \right)_{X_h}$ can be represented as
\begin{align}
w^T(K+M)v = w^TR^{-1}v,
\end{align}
and the dual pair $\langle r_h,v_h \rangle$ can be represented as $r^Tv$.

With the above preparation, We are ready to compare the CG iterates for both $AU=F$ and $\mathscr{L}_h u_h = f_h$. In the following, the CG iterates for $AU=F$ will be denoted as $U_k, r_k, p_k, \alpha_k, \beta_k$, the CG iterates for $\mathscr{L}_h u_h = f_h$ will be denoted as $u_h^k, r_h^k, p_h^k, \alpha^k, \beta^k$, and the representation of $u_h^k, r_h^k, p_h^k$ as $U^k, r^k, p^k$. The comparisons are sumerized in Table \ref{tab1:comparison-iterates-of-cg}. Taking a closer look at formulas for $p^k, \alpha^k, \beta^k$, we find that the representive CG iterates for $\mathscr{L}_h u_h = f_h$ is a preconditioned conjugate gradient method (PCG) for $AU=F$ with the preconditioner $R$, which is the discretized Riesz isomorphism. The CG for the discretized operator equation $\mathscr{L}_h u_h = f_h$ can be viewed as an inexact CG in function space applied to the operator equation $\mathscr{L} u = f$, just as the relationship between the finite dimensional Newton method and the infinite dimensional Newton method for nonlinear differential equations, see \cite{newtonmethod-Deuflhard}. In this sense, any PCG for $AU=F$ can be considered as an inexact infinite dimensional CG for the operator equation $\mathscr{L} u = f$.

\begin{table}[!h]
\renewcommand{\arraystretch}{1.10}
\tabcolsep 0pt \caption{Comparison of the CG iterates.}
\vspace*{-12pt} \label{tab1:comparison-iterates-of-cg}
\begin{center}
\def\temptablewidth{1.0\textwidth}
{\rule{\temptablewidth}{1pt}}
\begin{tabular*}{\temptablewidth}{@{\extracolsep{\fill}}ccccc}
$AU=F$ & & $\mathscr{L}_h u_h = f_h$ & & Representation of iterates\\
\cline{1-5}
\\
$U_0$ & & $u_h^0=Q_1U_0$ & & $u^0=U_0$
\\
\\
$r_0 = F-AU_0 = Q_2 r_h^0$ &  & $r_{h}^0=f_h-\mathscr{L}_{h} u_{h}^0$ & &$r^0 = Q_2 r_h^0 = r_0$
\\
\\$p_0 = r_0$ &  & $p_{h}^0=\mathscr{R}_h r_h^0$ & & $p^0=R r^0$
\\
\\
$\alpha_0=\frac{r_0^Tp_0}{p_0^TAp_0}$ & & $\alpha^0=\frac{\langle r_h^0, p_h^{0} \rangle}{\langle \mathscr{L}_h p_h^{0},p_h^{0} \rangle}$ & & $\alpha^0=\frac{(r^0)^T p^{0} }{(p^0)^TAp^0}$
\\
\\
$U_{1}=U_0+\alpha_0 p_0$ & & $u_h^{1}=u_h^0+\alpha^0 p_h^0$ & & $u^{1}=u^0+\alpha^0 p^0$
\\
\\
$r_{1}=r_0-\alpha_0 A p_0$ & & $r_h^{1}=r_h^0-\alpha^0 \mathscr{L}_h p_h^0$ & & $r^{1}=r^0-\alpha^0 A p^0$
\\
\\
$\beta_{k-1}=-\frac{p_{k-1}^TAr_k}{p_{k-1}^TAp_{k-1}}$ & & $\beta^{k-1}=-\frac{\langle \mathscr{L}_h p_h^{k-1},\mathscr{R}_h r_h^{k} \rangle}{\langle \mathscr{L}_h p_h^{k-1},p_h^{k-1} \rangle}$ & & $\beta^{k-1}=-\frac{(p^{k-1})^T A (Rr^{k}) }{(p^{k-1})^T A p^{k-1}}$
\\
\\
$p_k=r_k+\beta_{k-1} p_{k-1}$ & & $p_{h}^k=\mathscr{R}_{h}r_{h}^k+\beta^{k-1} p_{h}^{k-1}$ & & $p^k=R r^k+\beta^{k-1} p^{k-1}$
\\
\\
$\alpha_k=\frac{r_k^Tp_k}{p_k^TAp_k}$ & & $\alpha^k=\frac{\langle r_h^k, p_h^{k} \rangle}{\langle \mathscr{L}_h p_h^{k},p_h^{k} \rangle}$ & & $\alpha^k=\frac{(Rr^k)^Tp^k}{(p^k)^TAp^k}$
\\
\\
$U_{k+1}=U_k+\alpha_k p_k$ & & $u_h^{k+1}=u_h^k+\alpha^k p_h^k$ & & $U^{k+1}=U^k+\alpha^k p^k$
\\
\\
$r_{k+1}=r_k-\alpha_kAp_k$ & & $r_h^{k+1}=r_h^k-\alpha^k \mathscr{L}_h p_h^k$ & & $r^{k+1}=r^k-\alpha^k A p^k$
\\
\\
\end{tabular*}
{\rule{\temptablewidth}{1pt}}
\end{center}
\end{table}

\section{Concluding remarks}
\label{sec6:conclusions}

CG is one of connection nodes of computational mathematics. It connects to the conjugate direction method, the subspace optimization method and the BFGS quasi-Newton method in numerical optimization, connects to the Lanczos method in  numerical linear algebra, connects to the orthogonal polynomials in numerical approximation theory, and connects to PCG in numerical PDEs. It is full of mathematical ideas and novel computational techniques. Maybe there are still undiscovered connections or merits inside CG.


\end{document}